\documentclass{psp}
\usepackage{amsmath,amssymb}
\usepackage[dvips]{graphics}
\usepackage[all]{xy}
\usepackage[usenames]{color}

\usepackage{pifont, amsfonts}
\usepackage{mathtools}
\usepackage[allowmove]{url}
\usepackage{multirow}
\usepackage{fancyhdr,ifthen}
\usepackage{enumerate}
\usepackage{amsmath}
\usepackage{amsfonts}
\usepackage{graphicx}

\newtheorem{Theorem}{Theorem}[section]
\newtheorem{Lemma}[Theorem]{Lemma}
\newtheorem{Proposition}[Theorem]{Proposition}
\newtheorem{Corollary}[Theorem]{Corollary}
\newtheorem{Definition}[Theorem]{Definition}
\newtheorem{Example}[Theorem]{Example}

\newtheorem{Remark}[Theorem]{Remark}

\def\theenumi{\arabic{enumi}}

\def\theenumii{\roman{enumii}}

\numberwithin{equation}{section}

\begin{document}
\title[Reduced Wu and Generalized Simon Invariants for Spatial Graphs]{Reduced Wu and Generalized Simon Invariants for Spatial Graphs}
\author[E. Flapan, W.R. Fletcher and R. Nikkuni]
{ERICA FLAPAN\\
Department of Mathematics\addressbreak Pomona College\addressbreak Claremont, CA 91711, USA
\nextauthor WILL FLETCHER\\
Biophysics Program\addressbreak Stanford University\addressbreak Stanford, CA 94305, USA
\and\ RYO NIKKUNI\\
Department of Mathematics, \addressbreak
Tokyo Woman's Christian University, \addressbreak
2-6-1 Zempukuji, Suginami-ku, \addressbreak
Tokyo 167-8585, Japan
}

    
    




\maketitle

\begin{abstract}
	We introduce invariants of graphs embedded in $S^3$ which are related to the Wu invariant and the Simon invariant.  Then we use our invariants to prove that certain graphs are intrinsically chiral, and to obtain lower bounds for the minimal crossing number of particular embeddings of graphs in $S^3$.  \end{abstract}

\section{Introduction}

While there are numerous invariants for embeddings of graphs in $3$-manifolds, most have limited applications either because they are hard to compute or because they are only defined for particular types of graphs.  For example, Thompson \cite{Tho} defined a powerful polynomial invariant for graphs embedded in arbitrary 3-manifolds, which can detect whether an embedding of a graph in $S^3$ is planar.  However, computing Thompson's invariant requires identifying topological features of a sequence of 3-manifolds, such as whether each manifold is compressible.   

Yamada \cite{Ya} and Yokota \cite{Yo} introduced polynomial invariants for spatial graphs (i.e., graphs embedded in $S^3$).  The Yamada polynomial is an ambient isotopy invariant for spatial graphs with vertices of degree at most 3.  However, for other spatial graphs it is only a regular isotopy invariant.   It is convenient to use because  it can be computed using skein relations.  Also, the Yamada polynomial can be used to detect whether a spatial graph with vertices of degree at most 3 is chiral (i.e., distinct from its mirror image).  The Yokota polynomial is an ambient isotopy invariant for all spatial graphs that reduces to the Yamada polynomial for graphs with vertices of degree at most 3.  However, the Yokota polynomial is more difficult to compute, and cannot be used to show that a spatial graph is chiral.  

In a lecture in 1990, Jon Simon introduced an invariant of embeddings of the graphs $K_5$ and $K_{3,3}$ with labeled vertices in $S^3$.  The Simon invariant is easy to compute from a projection of an embedding and has been useful in obtaining results about embeddings of non-planar graphs \cite{Huh, Nikkuni 2006a, Nikkuni 2006, Nikkuni 2009, Nikkuni 2009b, Ohyama, Shinjo, Taniyama 1994, Taniyama 1995}.  In 1995, Taniyama \cite{Taniyama 1995} showed that the Simon invariant is a special case of a cohomology invariant for all spatial graphs which had been introduced by Wu \cite{Wu 1960, Wu 1965}, and showed that the Wu invariant can be defined combinatorially from a graph projection.  However, the Wu invariant is not always easy to compute, and (like the Simon invariant) depends on the choice of labeling of the vertices of a graph.  For this reason, the role of the Wu invariant in distinguishing a spatial graph from its mirror image has been limited to showing that for any embedded non-planar graph $\Gamma$, there is no orientation reversing homeomorphism of $(S^3,\Gamma)$ that fixes every vertex of $\Gamma$ (see \cite{Nikkuni 2006}). Without this restriction on the vertices, many non-planar graphs including $K_5$ and $K_{3,3}$ have achiral embeddings as shown in Figure~\ref{achiral}.

\begin{figure}[here]
\begin{center}
\includegraphics[width=0.5\textwidth]{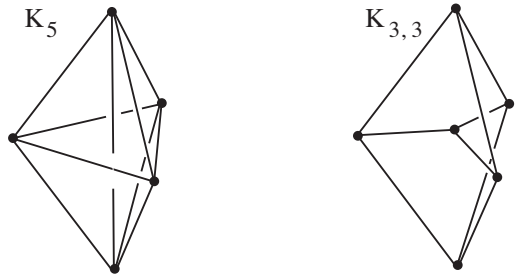}
\caption{Achiral embeddings of $K_5$ and $K_{3,3}$.}
\label{achiral}
\end{center}
\end{figure}

In this paper, we define numerical invariants that are obtained by reducing the Wu invariant and by generalizing the Simon invariant.  We then use our invariants to prove that no matter how the complete graph $K_7$, the M\"{o}bius ladders $M_{2N+1}$, and the Heawood graph are embedded in $S^3$, there is no orientation reversing homeomorphism of $S^3$ which takes the embedded graph to itself.  Finally, we show that our invariants can be used to give a lower bound on the minimal crossing number of embedded graphs.

\section{Wu Invariants and Reduced Wu Invariants}

In 1960, Wu \cite{Wu 1960} introduced an invariant as follows. Let $C_{2}(X)$ be the {\it configuration space} of ordered pairs of points from a topological space $X$, namely
\begin{eqnarray*}
C_{2}(X) = \left\{(x,y)\in X\times X~|~x\neq y\right\}. 
\end{eqnarray*}
Let $\sigma$ be the involution of $C_{2}(X)$ given by $\sigma(x,y)=(y,x)$. The integral cohomology group of ${\rm
Ker}\left(1+\sigma_{\sharp}\right)$ denoted by
$H^{*}\left(C_{2}(X),\sigma\right)$ is said to be the {\it
skew-symmetric integral cohomology group} of the pair
$\left(C_{2}(X),\sigma\right)$, where $\sigma_{\sharp}$ denotes the
chain map induced by $\sigma$.  Wu  \cite{Wu 1960} proved that $H^{2}(C_{2}({\mathbb R}^{3}),\sigma)\cong {\mathbb Z}$, and hence is generated by some element $\Sigma$.  Let $f:G\to {\mathbb R}^{3}$ be a spatial embedding of a graph $G$ with labeled vertices and  orientations on the edges. Then $f$ naturally induces an equivariant embedding $f\times f:C_{2}(G)\to C_{2}({\mathbb R}^{3})$ with respect to the action $\sigma$, and therefore induces a homomorphism 
\begin{eqnarray*}
(f\times f)^{*}:H^{2}(C_{2}({\mathbb R}^{3}),\sigma)
\longrightarrow H^{2}(C_{2}(G),\sigma). 
\end{eqnarray*}
The element $(f\times f)^{*}(\Sigma)$ is an ambient isotopy invariant known as the {\it Wu invariant}.   

In order to explicitly calculate the Wu invariant, Taniyama \cite{Taniyama 1995} developed the following combinatorial approach.  Let $G$ be a graph with vertices labeled $v_{1},v_{2},\ldots,v_{m}$ and oriented edges labeled $e_{1},e_{2},\ldots,e_{n}$. For each pair of disjoint edges $e_{i}$ and $e_{j}$, we define a variable $E^{e_{i},e_{j}}=E^{e_{j},e_{i}}$; and for each edge $e_{i}$ and vertex $v_{s}$ which is disjoint from $e_i$, we define a variable $V^{e_{i},v_{s}}$. Let $Z(G)$ be the free ${\mathbb Z}$-module generated by the collection of $E^{e_{i},e_{j}}$'s. For each $V^{e_{i},v_{s}}$, let $\delta(V^{e_{i},v_{s}})$ be the element of $Z(G)$ given by the sum of all $E^{e_{i},e_{k}}$ such that $e_k$ is disjoint from $e_i$ and has initial vertex $v_s$, minus the sum of all $E^{e_{i},e_{k}}$ such that $e_k$ is disjoint from $e_i$ and has terminal vertex $v_s$.  Thus

\begin{eqnarray*}
\delta(V^{e_{i},v_{s}})
= 
\sum_{\substack{
{I(k)=s} 
\\ e_{i}\cap e_{k}=\emptyset
}
}E^{e_{i},e_{k}}
-
\sum_{\substack{
{T(l)=s} 
\\ e_{i}\cap e_{l}=\emptyset
} 
}E^{e_{i},e_{l}}, 
\end{eqnarray*}

\noindent where $I(k)=s$ indicates that the initial vertex of $e_{k}$ is $v_{s}$, and $T(l)=s$ indicates that the terminal vertex of $e_{l}$ is $v_{s}$. Let $B(G)$ be the submodule of $Z(G)$ generated by the collection of $\delta(V^{e_{i},v_{s}})$'s. We let $L(G)$ denote the quotient module $Z(G) / B(G)$, and call it a {\it linking module} of $G$. Then $L(G)\cong H^{2}(C_{2}(G),\sigma)$. 

Now let $f$ be an embedding of the labeled oriented graph $G$ in $S^3$.  Fix a projection of $f(G)$ and let $\ell(f(e_{i}),f(e_{j}))=\ell(f(e_{j}),f(e_{i}))$ denote the sum of the signs of the crossings between $f(e_{i})$ and $f(e_{j})$. Taniyama \cite{Taniyama 1995} showed  that the equivalence class 
\begin{eqnarray*}
{\mathcal L}(f)=
\left[
\sum_{e_{i}\cap e_{j}=\emptyset}\ell(f(e_{i}),f(e_{j}))E^{e_{i},e_{j}}
\right]
\in L(G)
\end{eqnarray*}
coincides with $(f\times f)^{*}(\Sigma)$ through the isomorphism from $H^{2}(C_{2}(G),\sigma)$ to $L(G)$. Thus we may regard ${\mathcal L}(f)$ as the Wu invariant of $f$. Furthermore, $H^{2}(C_{2}(G),\sigma)$ is torsion free, namely $L(G)$ is a free ${\mathbb Z}$-module, and for an orientation-reversing self-homeomorphism $\Phi$ of $S^{3}$, it follows that ${\mathcal L}(\Phi\circ f)=-{\mathcal L}(f)$.

\begin{figure}[htbp]
     
\includegraphics[width=\textwidth]{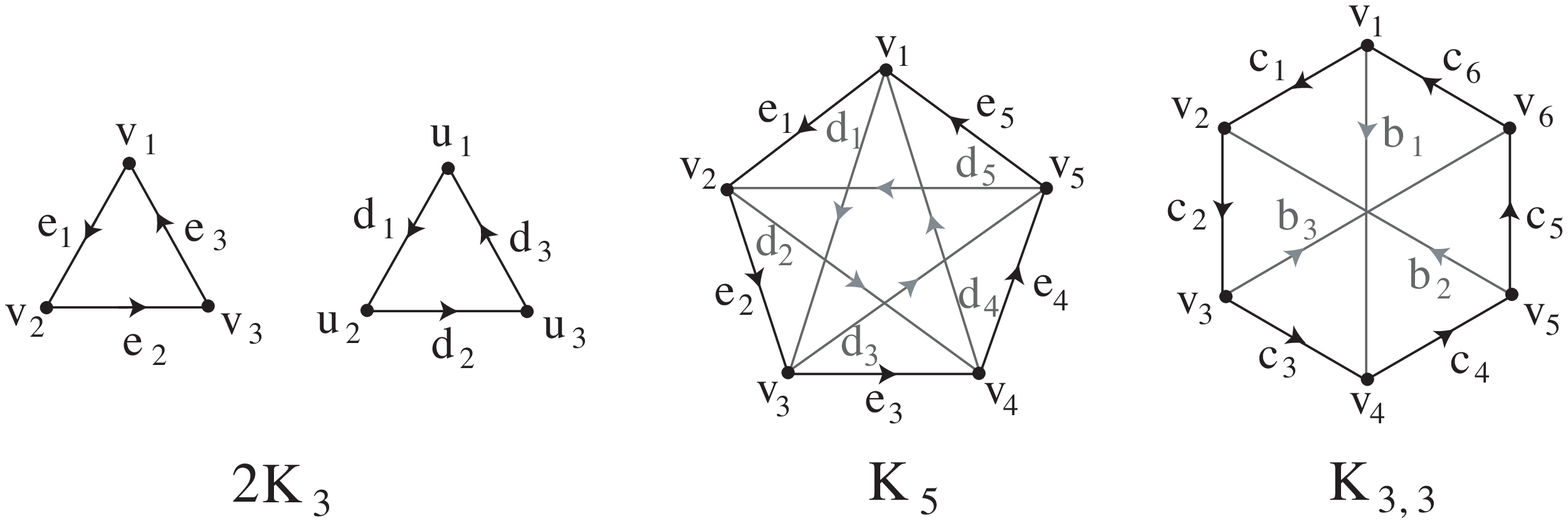}
   
   \caption{The Wu invariants for embeddings of these graphs are given in Examples \ref{2K3}, \ref{K5}, and \ref{K33}.}
  \label{2K3K5K33}
\end{figure} 

\begin{Example}\label{2K3}
{\rm 
Let $2K_{3}$ denote the graph consisting of two copies of $K_{3}$, labeled and oriented as illustrated in Figure~\ref{2K3K5K33}, and let $f$ be a spatial embedding of $2K_{3}$.  It was shown in \cite{Taniyama 1995} that the linking module $L(2K_{3})=\langle [E^{e_{1},d_{1}}]\rangle \cong {\mathbb Z}$, and the Wu invariant of $f$ is given by:
\begin{eqnarray*}
{\mathcal L}(f)
=\sum_{1\le i,j\le 3}\ell(f(e_{i}),f(d_{j}))[E^{e_{1},d_{1}}]
=2{\rm lk}(f)[E^{e_{1},d_{1}}], 
\end{eqnarray*}
where ${\rm lk}(f)$ denotes the {\it linking number} of the pair of triangles in $S^{3}$. 
}
\end{Example}
\medskip

\begin{Example}\label{K5}
{\rm 
Let $K_{5}$ denote the complete graph on five vertices, labeled and oriented as illustrated in Figure~\ref{2K3K5K33}, and let $f$ be a spatial embedding of $K_{5}$. It was shown in \cite{Taniyama 1995} that the linking module $L(K_{5})=\langle [E^{e_{1},e_{3}}]\rangle \cong {\mathbb Z}$ and the Wu invariant is given by:
\begin{eqnarray*}
{\mathcal L}(f)
=\sum_{a\cap b=\emptyset}\varepsilon(a,b)\ell(f(a),f(b))[E^{e_{1},e_{3}}], 
\end{eqnarray*}
where $\varepsilon(a,b)$ is defined by $\varepsilon(e_{i},e_{j})=1$, $\varepsilon(d_{i},d_{j})=-1$ and $\varepsilon(e_{i},d_{j})=-1$. 
}

\end{Example}
\medskip

 We work out the following example which it is given in \cite{Taniyama 1995} without details.

\begin{Example}\label{K33}
{\rm 
Let $K_{3,3}$ denote the complete bipartite graph, labeled and oriented as illustrated in Figure~\ref{2K3K5K33}, and let $f$ be a spatial embedding of $K_{3,3}$. Then $Z(K_{3,3})$ is a free ${\mathbb Z}$-module generated by 
\begin{eqnarray*}
&&E^{c_{1},c_{3}},E^{c_{2},c_{4}},E^{c_{3},c_{5}},E^{c_{4},c_{6}},E^{c_{5},c_{1}},E^{c_{6},c_{2}},E^{c_{1},c_{4}},E^{c_{2},c_{5}},E^{c_{3},c_{6}},\\
&&E^{b_{1},c_{2}},E^{b_{1},c_{5}},E^{b_{3},c_{4}},E^{b_{3},c_{1}},E^{b_{2},c_{3}},E^{b_{2},c_{6}},E^{b_{1},b_{2}},E^{b_{2},b_{3}},E^{b_{3},b_{1}}
\end{eqnarray*} 
and $B(K_{3,3})$ is a submodule of $Z(K_{3,3})$ generated by 
\begin{eqnarray*}
&&E^{b_{1},c_{2}}-E^{c_{6},c_{2}},\ E^{c_{1},c_{3}}-E^{c_{3},c_{6}},\ E^{c_{1},c_{4}}-E^{c_{4},c_{6}},\\
&&E^{c_{5},c_{1}}+E^{b_{1},c_{5}},\ E^{b_{3},c_{1}}+E^{b_{3},b_{1}},\ E^{b_{1},b_{2}}-E^{b_{2},c_{6}},\\
&&-E^{b_{2},c_{3}}-E^{c_{1},c_{3}},\ E^{c_{2},c_{4}}-E^{c_{1},c_{4}},\ E^{c_{2},c_{5}}-E^{c_{5},c_{1}},\\
&&E^{c_{6},c_{2}}-E^{b_{2},c_{6}},\ -E^{b_{2},b_{3}}-E^{b_{3},c_{1}},\ E^{b_{1},c_{2}}-E^{b_{1},b_{2}},\\
&&E^{b_{3},c_{4}}-E^{c_{2},c_{4}},\ E^{c_{3},c_{5}}-E^{c_{2},c_{5}},\ E^{c_{3},c_{6}}-E^{c_{6},c_{2}},\\
&&E^{c_{1},c_{3}}+E^{b_{3},c_{1}},\ E^{b_{2},c_{3}}+E^{b_{2},b_{3}},\ E^{b_{3},b_{1}}-E^{b_{1},c_{2}},\\
&&-E^{b_{1},c_{5}}-E^{c_{3},c_{5}},\ E^{c_{4},c_{6}}-E^{c_{3},c_{6}},\ E^{c_{1},c_{4}}-E^{c_{1},c_{3}},\\
&&E^{c_{2},c_{4}}-E^{b_{1},c_{2}},\ -E^{b_{1},b_{2}}-E^{b_{2},c_{3}},\ E^{b_{3},c_{4}}-E^{b_{3},b_{1}},\\
&&E^{b_{2},c_{6}}-E^{c_{4},c_{6}},\ E^{c_{5},c_{1}}-E^{c_{1},c_{4}},\ E^{c_{2},c_{5}}-E^{c_{2},c_{4}},\\
&&E^{c_{3},c_{5}}+E^{b_{2},c_{3}},\ E^{b_{1},c_{5}}+E^{b_{1},b_{2}},\ E^{b_{2},b_{3}}-E^{b_{3},c_{4}},\\
&&-E^{b_{3},c_{1}}-E^{c_{5},c_{1}},\ E^{c_{6},c_{2}}-E^{c_{2},c_{5}},\ E^{c_{3},c_{6}}-E^{c_{3},c_{5}},\\
&&E^{c_{4},c_{6}}-E^{d_{3},c_{4}},\ -E^{b_{3},b_{1}}-E^{b_{1},c_{5}},\ E^{b_{2},c_{6}}-E^{b_{2},b_{3}}. 
\end{eqnarray*}
Then we have 
\begin{eqnarray*}
&& [E^{c_{1},c_{3}}]=[E^{c_{2},c_{4}}]=[E^{c_{3},c_{5}}]=[E^{c_{4},c_{6}}]=[E^{c_{5},c_{1}}]=[E^{c_{6},c_{2}}]\\
&=& [E^{c_{1},c_{4}}]=[E^{c_{2},c_{5}}]=[E^{c_{3},c_{6}}]=[E^{b_{1},b_{2}}]=[E^{b_{2},b_{3}}]=[E^{b_{3},b_{1}}]\\
&=& [E^{b_{1},c_{2}}]=[E^{b_{3},c_{4}}]=[E^{b_{2},c_{6}}]\\
&=& -[E^{b_{1},c_{5}}] = -[E^{b_{3},c_{1}}] = -[E^{b_{2},c_{3}}], 
\end{eqnarray*} 
Then the linking module $L(K_{3,3})=\langle [E^{c_{1},c_{3}}]\rangle \cong {\mathbb Z}$ and the Wu invariant is given by: 
\begin{eqnarray*}
{\mathcal L}(f)
=\sum_{a\cap b=\emptyset}\varepsilon(a,b)\ell(f(a),f(b))[E^{c_{1},c_{3}}], 
\end{eqnarray*}
where $\varepsilon(a,b))$ is defined by 
$\varepsilon(c_{i},c_{j})=1$, $\varepsilon(b_{i},b_{j})=1$, and

\begin{equation*}
	\varepsilon(c_i,b_j) = 
	\begin{cases}
		1 &\text{if } c_i \text{ and } b_j \text { are parallel in Figure~\ref{2K3K5K33}}
\\
		-1 & \text{if } c_i \text{ and } b_j \text { are anti-parallel in Figure~\ref{2K3K5K33}}
		\end{cases}
\end{equation*}}
\end{Example}
\medskip

\begin{Remark}\rm{

 It was shown in \cite{Taniyama 1995} that $L(G)=0$ if and only if $G$ is a planar graph which does not contain a pair of two disjoint cycles. 
 }\end{Remark}

 \begin{Remark}
 \rm{
 It was shown in \cite{N00} that if the graph $G$ is {\it $3$-connected}, then 
\begin{eqnarray*}
{\rm rank}L(G)=\frac{1}{2}
\left\{
\beta_{1}(G)^{2}+\beta_{1}(G)+4|E(G)|-\sum_{v\in V(G)}\left({\rm deg}(v)\right)^{2}
\right\}, 
\end{eqnarray*}
where $\beta_{1}(G)$ denotes the first Betti number of $G$ and ${\rm deg}(v)$ denotes the valency of a vertex $v$. For example, ${\rm rank}(L(K_{6}))=10$ and ${\rm rank}(L(K_{7}))=36$. 
}
\end{Remark}

\medskip
\begin{Definition}
Let $f$ be a spatial embedding of an oriented graph $G$ with linking module $L(G)$ and Wu invariant ${\mathcal L}(f)\in L(G)$. Let $\varepsilon:L(G)\to {\mathbb Z}$ be a homomorphism. Then we call the integer $\varepsilon({\mathcal L}(f))$ the {\bf reduced Wu invariant of $f$ with respect to $\varepsilon$} and denote it by $\tilde{\mathcal L}_{\varepsilon}(f)$. 
\end{Definition}

For a pair of disjoint edges $e_i$ and $e_j$, we denote $\varepsilon([E^{e_{i},e_{j}}])$ by $\varepsilon(e_i,e_j)$.  Thus

\begin{eqnarray*}
\tilde{\mathcal L}_{\varepsilon}(f)=\varepsilon\left(\left[
\sum_{e_{i}\cap e_{j}=\emptyset}\ell(f(e_{i}),f(e_{j}))E^{e_{i},e_{j}}
\right]\right)=\sum_{e_{i}\cap e_{j}=\emptyset}\ell(f(e_{i}),f(e_{j}))\varepsilon(e_i,e_j).
\end{eqnarray*}

\medskip

\begin{Example}\label{2K3_2}
{\rm 
Consider $2K_{3}$, labeled and oriented as in Figure~\ref{2K3K5K33}, and let $f$ be an embedding of $2K_{3}$ in $S^3$. Let $\varepsilon$ be the isomorphism from $L(2K_{3})$ to ${\mathbb Z}$ defined by $\varepsilon(e_{1},d_{1})=1$. Then by Example \ref{2K3}, we have $\tilde{\mathcal L}_{\varepsilon}(f)=2{\rm lk}(f)$. 
}
\end{Example}
\medskip

\begin{Example}\label{K5K33_2}
{\rm 
Let $G$ be $K_{5}$ or $K_{3,3}$ labeled and oriented as illustrated in Figure~\ref{2K3K5K33}, and let $f$ be an embedding of $G$ in $S^3$.  Let  $\varepsilon$ be the isomorphism from $L(G)$ to ${\mathbb Z}$ defined by   $\varepsilon(e_1,e_3)=1$ for $G=K_5$ and $\varepsilon(c_1, c_3)=1$ for $G=K_{3,3}$.  Then it follows that $\tilde{\mathcal L}_{\varepsilon}(f)=\sum_{a\cap b=\emptyset}\varepsilon(a,b)\ell(f(a),f(b))$, where the value of $\varepsilon(a,b)$ for an arbitrary pair of edges is given in Example \ref{K5} if $G=K_{5}$ and in Example \ref{K33} if $G=K_{3,3}$. }
\end{Example}

\begin{figure}[htbp]
      \begin{center}
\scalebox{0.425}{\includegraphics{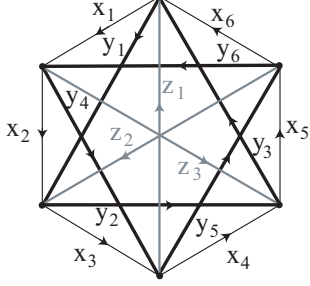}}
      \end{center}
   \caption{A reduced Wu invariant for $K_6$ is given in Example \ref{K6gs}.}
  \label{K6Simon}
\end{figure} 

\begin{Example}\label{K6gs}
{\rm Consider $K_{6}$, labeled and oriented as in Figure~\ref{K6Simon}, and let $f$ be an embedding of $K_{6}$ in $S^3$. For any pair of disjoint edges $a$ and $b$ in $K_6$, we define $\varepsilon(a,b)$ as follows:

\begin{equation*}
	\varepsilon(x_i,x_j) = 
	\begin{cases}
		3 &\text{if } x_i \text{ and } x_j \text { are anti-parallel in Figure~\ref{K6Simon}}
\\
		2 & \text{if } x_i \text{ and } x_j \text { are neither parallel nor anti-parallel in Figure~\ref{K6Simon}}
		\end{cases}
\end{equation*}

\medskip

\begin{equation*}
	\varepsilon(y_i,y_j) = 
	\begin{cases}
		0 &\text{if } y_i \text{ and } y_j \text { are anti-parallel in Figure~\ref{K6Simon}}
\\
		-1 & \text{if } y_i \text{ and } y_j \text { are neither parallel nor anti-parallel in Figure~\ref{K6Simon}}
		\end{cases}
\end{equation*}

\medskip

\begin{equation*}
	\varepsilon(x_i,z_j) = 
	\begin{cases} 		-1 & \text{if } x_i \text{ and } z_j \text { are anti-parallel in Figure~\ref{K6Simon}}
		
\\
1 &\text{if } x_i \text{ and } z_j \text { are parallel in Figure~\ref{K6Simon}}
		\end{cases}
\end{equation*}
\medskip

\noindent In addition, we define $\varepsilon(z_{i},z_{j})=1$, $\varepsilon(x_{i},y_{j})=-1$, and $\varepsilon(y_{i},z_{j})=0$. Then it can be checked that $\varepsilon$ gives a homomorphism from $L(K_{6})$ to ${\mathbb Z}$. It follows that $\tilde{\mathcal L}_{\varepsilon}(f)=\sum_{a\cap b=\emptyset}\varepsilon(a,b)\ell(f(a),f(b))$ is a reduced Wu invariant for $K_6$. 
}
\end{Example}

\section{Generalized Simon Invariants}\label{GSI}
Simon introduced the following function of embeddings $f$ of the  graphs $K_5$ and $K_{3,3}$, labeled and oriented as in Figure \ref{2K3K5K33}.  Let

\[
\widehat{L}_{\varepsilon}(f) = \sum_{a\cap b=\emptyset} \varepsilon(a,b) \ell(f(a),f(b))
\]

\noindent where $\varepsilon(a,b)$ is defined as $\varepsilon(e_{i},e_{j})=1$, and $\varepsilon(d_{i},d_{j})=\varepsilon(e_{i},d_{j})=-1$ for $K_5$; and $\varepsilon(a,b)$ is defined as $\varepsilon(c_{i},c_{j})=1$, $\varepsilon(b_{i},b_{j})=1$
\begin{equation*}
	\varepsilon(c_i,b_j) = 
	\begin{cases}
		1 &\text{if } c_i \text{ and } b_j \text { are parallel in Figure~\ref{2K3K5K33}}
\\
		-1 & \text{if } c_i \text{ and } b_j \text { are anti-parallel in Figure~\ref{2K3K5K33}}
		\end{cases}
\end{equation*}
\noindent  for $K_{3,3}$.

 Simon then proved that for any projection of an embedding $f$ of the oriented labeled graphs $K_5$ and $K_{3,3}$, the value  of \[
\sum_{a,b \in G} \varepsilon(a,b) \ell(f(a),f(b))
\]

\noindent is invariant under the five Reidemeister moves for spatial graphs given in Figure~\ref{Figuremy5radmoves}.  This invariant is known as the {\it Simon invariant}.

\begin{figure}[here]
\begin{center}
\includegraphics[width=0.6\textwidth]{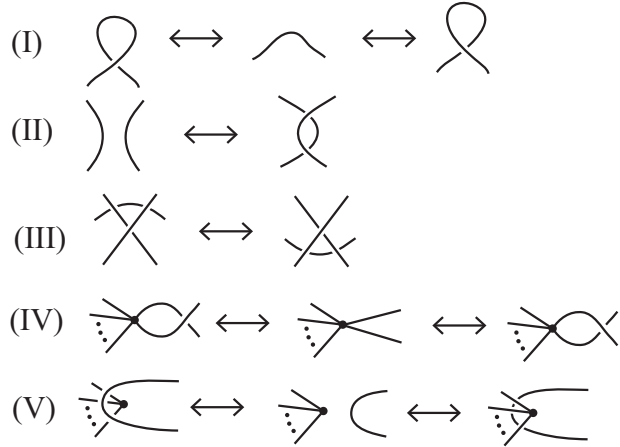}
\caption{The Reidemeister moves for embedded graphs.}
\label{Figuremy5radmoves}
\end{center}
\end{figure}

By using Simon's method we can create similar invariants for many other embedded graphs.  In particular, let $G$ be an oriented graph and let $f$ be an embedding of $G$ in $S^3$. If we can define a function $\varepsilon(a,b)$ from the set of pairs of disjoint edges of $G$ to the integers such that for any projection of $f(G)$ the value of

\[
\widehat{L}_{\varepsilon}(f) = \sum_{a\cap b=\emptyset} \varepsilon(a,b) \ell(f(a),f(b))
\]
\noindent is invariant under the five Reidemeister moves, then we say that $\widehat{L}_{\varepsilon}(f)$ is a {\it generalized Simon invariant of $f(G)$}.  If for every embedding $f$ of $G$, $\widehat{L}_{\varepsilon}(f)$ is a {\it generalized Simon invariant of $f(G)$}, then we say that $\widehat{L}_{\varepsilon}(f)$ is a {\it generalized Simon invariant of $G$.}
\medskip

Observe that the reduced Wu invariants given in Example \ref{K5K33_2} are identical to their Simon invariants.  In fact, every reduced Wu invariant with respect to a given homomorphism $\varepsilon$ is a generalized Simon invariant with epsilon coefficients given by $\varepsilon(a,b)$.  However, not every generalized Simon invariant is necessarily a reduced Wu invariant.  In order to distinguish these two types of invariants, we use $\tilde{\mathcal L}_{\varepsilon}(f)$ to denote a reduced Wu invariant and $\widehat{L}_{\varepsilon}(f)$ to denote a generalized Simon invariant.

We say that a graph embedded in ${S}^3$ is \textit{achiral} if there is an orientation reversing homeomorphism of ${S}^3$ that takes the graph to itself setwise.  Otherwise, we say the embedded graph is \textit{chiral}.  We say that an abstract graph is \textit{intrinsically chiral} if every embedding of the graph in $S^3$ is chiral.  Note that when we talk about chirality or achirality we are considering embedded graphs as subsets of $S^3$ disregarding any edge labels or orientations.  For example, we saw in Figure~\ref{achiral} that $K_5$ and $K_{3,3}$ have achiral embeddings, although it was shown in \cite{Nikkuni 2006} that no embedding of either of these graphs has an orientation reversing homeomorphism that preserves the edge labels and orientations given in Figure~\ref{2K3K5K33}.

We now define generalized Simon invariants for some specific graphs and families of graphs, and use these invariants to prove that the graphs are intrinsically chiral.

\medskip
\subsection*{The complete graph $K_7$}\label{sec:K7}
Consider the complete graph $K_7$ with labeled edges as illustrated in Figure~\ref{FigureK7_labeled}.  We refer to the edges $x_1, x_2, ..., x_7$ as ``outer edges''  and the rest of the edges as ``inner edges.''  We refer to the Hamiltonian cycle $\overline{y_1 y_2 ... y_7}$ as the {\it 1-star} since these edges skip over one vertex relative to the cycle $\overline{x_1x_2...x_7}$. Similarly, we refer to the Hamiltonian cycle $\overline{z_1 z_2 ... z_7}$ as the {\it 2-star} since these edges skip over two vertices relative to the cycle $\overline{x_1x_2...x_7}$.  For consistency, we also use the term {\it 0-star} to refer to the Hamiltonian cycle $\overline{x_1x_2...x_7}$.  We orient the edges around each of the stars as illustrated.  Note that this classification of oriented edges is only dependent on our initial choice of an oriented 0-star.

\begin{figure}[here]
\begin{center}
\includegraphics[width=0.35\textwidth]{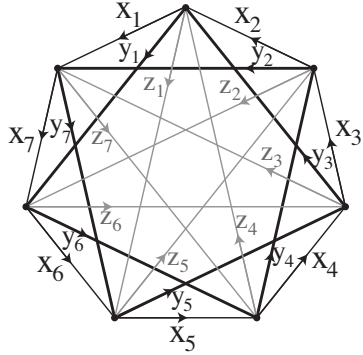}
\caption{An illustration of the oriented $K_7$, with the 0-star in black, the 1-star in bold, and the 2-star in grey.}
\label{FigureK7_labeled}
\end{center}
\end{figure}

We define the {\it epsilon coefficient} of a pair of disjoint edges by the function:

\[	\varepsilon(x_i,x_j) = \varepsilon(y_i,y_j) = \varepsilon(z_i,z_j) = \varepsilon(x_i,z_j) = \varepsilon(y_i,z_j) = 1 	\]
\[	\varepsilon(x_i,y_j) = -1.	\]

Given an oriented 0-star and an embedding $f: K_7 \rightarrow {S}^3$ with a regular projection, we define the integer $\widehat{L}_{\varepsilon}(f)$ by
\medskip

\[
\widehat{L}_{\varepsilon}(f) = \sum_{a\cap b=\emptyset} \varepsilon(a,b) \ell(f(a),f(b)).
\]

\medskip

\begin{Lemma}\label{k7 invariant}  Consider $K_7$ with a fixed choice of an oriented $0$-star. Then for any embedding $f:K_7\to S^3$, the value of $\widehat{L}_{\varepsilon}(f)$ is an ambient isotopy invariant.
\end{Lemma}

\begin{proof}  It is easy to check that $\widehat{L}_{\varepsilon}(f)$ is invariant under the first four Reidemeister moves.

 In order to show that $\widehat{L}_{\varepsilon}(f)$ is invariant under the fifth move, we must show that the value is unchanged when any edge of $f(K_7)$ is pulled over or under a given vertex $v$.  An example is illustrated in Figure~\ref{FigureK7_loopover}.  

\begin{figure}[here]
\begin{center}
\includegraphics[width=0.55\textwidth]{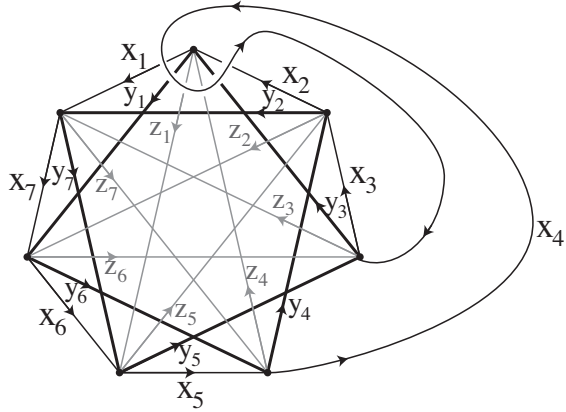}
\caption{$K_7$ with an edge pulled over a vertex.}
\label{FigureK7_loopover}
\end{center}
\end{figure}

Pulling a given edge $e$ over a vertex will generate six new crossings. In Figure~\ref{FigureK7_loopover} the edge $x_4$ has new crossings with the edges $x_1$, $y_1$, $z_1$, $x_2$, $y_3$, and $z_4$.  The crossings with edges pointed away from the vertex $v$ ($x_1$, $y_1$, $z_1$) will have an opposite sign compared to the crossings with edges pointed toward the vertex $v$ ($x_2$, $y_3$, $z_4$).  Thus, the overall change in $\widehat{L}_{\varepsilon}(f)$ is found by adding the epsilon coefficients for the crossings of $x_4$ with $x_1$, $y_1$, and $z_1$ and subtracting the epsilon coefficients for the crossings of $x_4$ with $x_2$, $y_3$, and $z_4$.  It is easy to check that $\widehat{L}_{\varepsilon}(f)$ is unchanged in each case.  \end{proof}
\medskip

It follows from Lemma~\ref{k7 invariant}, that $\widehat{L}_{\varepsilon}(f)$ is a generalized Simon invariant.  

\begin{Remark}\rm{ One can check that the epsilon coefficients we have given for $K_7$ define a homomorphism from the free $\mathbb{Z}$-module $L(K_7)$ to $\mathbb{Z}$.  Thus $\widehat{L}_{\varepsilon}(f)$ also gives us a reduced Wu invariant for $K_7$.}
\end{Remark}
\medskip

 We now apply the generalized Simon invariant of $K_7$ to prove that $K_7$ is intrinsically chiral.  This result was previously proven by Flapan and Weaver~\cite{Flapan 1992}, but using the generalized Simon invariant allows us to give a simpler proof which can be generalized to apply to many other graphs.  We begin with a lemma.



\begin{Lemma}\label{k7 odd}
For any embedding $f$ of  $K_7$ in ${S}^3$, the generalized Simon invariant $\widehat{L}_{\varepsilon}(f)$ is an odd number.
\end{Lemma}

\begin{proof}
Since any crossing change will change the signed crossing number between two edges by $\pm 2$, we only need to find an embedding $f$ where $\widehat{L}_{\varepsilon}(f)$ is odd.  Consider an embedding of $K_7$ which has Figure~\ref{FigureK7_labeled} as its projection with the intersections between edges replaced by crossings.  Note that there are 35 crossings in this embedding of $K_7$: 14 crossings of the 2-star with itself, and 21 crossings between the 1-star and the 2-star.  The epsilon coefficient for every one of these crossings is 1.  Since there is an odd number of crossings, regardless of their signs, $\widehat{L}_{\varepsilon}(f)$ must be odd.  Because any crossing change will change $\widehat{L}_{\varepsilon}(f)$ by an even number, it follows that $\widehat{L}_{\varepsilon}(f)$ is odd for any embedding of $K_7$.
\end{proof}

\medskip

\begin{Theorem}\label{k7 chiral}
$K_7$ is intrinsically chiral.
\end{Theorem}

\begin{proof}  For the sake of contradiction, suppose that for some embedding $f$ of $K_7$ there is an orientation reversing homeomorphism $h$ of the pair (${S}^3$, $f(K_7)$).  Let $\alpha$ denote the automorphism of $K_7$ that is induced by $h$. 

 Let $J$ denote the set of Hamiltonian cycles in $f(K_7)$ with non-zero Arf invariant.  Since any homeomorphism of ${S}^3$ preserves the Arf invariant of a knot, the homeomorphism $h$ permutes the elements of $J$.  It follows from Conway and Gordon~\cite{Conway 1983} that $|J|$ must be odd, and hence there is an orbit $O$ in $J$ such that $|O| = n$ for some odd number $n$.  Consequently, $h^n$ setwise fixes an element of $O$.  Hence some Hamiltonian cycle $C$ with non-zero Arf invariant is setwise fixed by $h^n$.  We now label and orient the edges of $K_7$ as in Figure~\ref{FigureK7_labeled} so that $f$ takes the 0-star of $K_7$ to $C$.  Since $h^n$ leaves $C$ setwise invariant, the automorphism $\alpha^ n$ (induced on $K_7$ by $h^n$) leaves the 0-star, 1-star, and 2-star all setwise invariant.  
 
 Fix a sphere of projection $P$ in $S^3$.  Since $f\circ \alpha^n(K_7)$ and $f(K_7)$ are identical as subsets of $S^3$, their projections on $P$ are the same.   Furthermore, if $\alpha^n$ preserves the orientation of the 0-star, then $\alpha^n$ preserves the orientation of the 1-star and 2-star and hence of every edge.  Otherwise, $\alpha^n$ reverses the orientation of every edge.  In either case, a given crossing in the projection has the same sign whether it is considered with orientations induced by $\alpha^n(K_7)$ or with orientations induced by $K_7$.  Furthermore, since $\alpha^n$ leaves the 0-star, 1-star, and 2-star of $K_7$ setwise invariant, each crossing has the same epsilon coefficient, whether the crossing is considered in $f\circ \alpha^n(K_7)$ or in $f(K_7)$.  It follows that $\widehat{L}_{\varepsilon}(h^n\circ f)=\widehat{L}_{\varepsilon}(\alpha^n\circ f)=\widehat{L}_{\varepsilon}(f)$.
 
Let $\rho$ denote a reflection of $S^3$ which pointwise fixes the sphere of projection $P$.  Using orientations induced by $K_7$, we see that the sign of every crossing in the projection of $\rho\circ f(K_7)$ on $P$ is the reverse of that of the corresponding crossing in the projection of $f(K_7)$.  Using the oriented 0-star from $K_7$, it follows that $\widehat{L}_{\varepsilon}(\rho\circ f)=-\widehat{L}_{\varepsilon}(f)$. On the other hand, since $n$ is odd $h^n$ is orientation reversing and is thus isotopic to $\rho$.  Hence by Lemma~\ref{k7 invariant},  $\widehat{L}_{\varepsilon}(\rho\circ f)=\widehat{L}_{\varepsilon}(h^n\circ f)$.  Consequently, $\widehat{L}_{\varepsilon}(h^n\circ f)=-\widehat{L}_{\varepsilon}(f)$.   Thus $\widehat{L}_{\varepsilon}(f)=0$, which contradicts Lemma~\ref{k7 odd}.  Hence in fact, $K_7$ is intrinsically chiral.\end{proof} 

\medskip

 \begin{Corollary}\label{K4n+3} For every odd number $n$, the complete graph $K_{4n+3}$ is intrinsically chiral.\end{Corollary}

\begin{proof}  Suppose that for some embedding $f$ of $K_{4n+3}$ in $S^3$, there is an orientation reversing homeomorphism $h$ of (${S}^3$, $f(K_{4n+3})$).  Even though in general the homeomorphism $h$ will not have finite order, the automorphism that $h$ induces on $K_{4n+3}$ does have finite order and its order can be expressed as $2^ab$ for some odd number $b$.  Now $g=h^b$ is an orientation reversing homeomorphism of (${S}^3$, $f(K_{4n+3})$) which induces
an automorphism of $K_{4n+3}$ of order $2^a$.

Observe that the number of $K_7$ subgraphs in $K_{4n+3}$ is 

$$\frac{(4n+3)(4n+2)(4n+1)(4n)(4n-1)(4n-2)(4n-3)}{7!}$$
$$=\frac{(4n+3)(2n+1)(4n+1)(n)(4n-1)(2n-1)(4n-3)}{315}.$$

\medskip

 \noindent This number is odd, since $n$ is odd.  Thus $g$ leaves invariant some $K_7$ subgraph.  But this is impossible since by Theorem \ref{k7 chiral}, $K_7$ is intrinsically chiral.\end{proof}

\medskip

\subsection*{Mobius ladders}
\label{sec:mobius}

A M\"{o}bius ladder $M_n$ with $n$ rungs is the graph obtained from a circle with $2n$ vertices by adding an edge between every pair of antipodal vertices.  Let $N\geq 2$, and consider the oriented labeled graph of $M_{2N+1}$ illustrated in Figure~\ref{FigureM_2N+1} (note there is no vertex at the center of the circle).  We denote the ``outer edges'' consecutively as $x_1, x_2, ..., x_{2(2N+1)}$, and the ``inner edges'' consecutively as $y_1, y_2, ..., y_{2N+1}$.  Since $N \geq 2$, it follows from Simon~\cite{S86} that there is no automorphism of $M_{2N+1}$ which takes an outer edge to an inner edge.  Thus, the distinction between inner and outer edges does not depend on any particular labeling.

\begin{figure}[here]
\begin{center}
\includegraphics[width=0.35\textwidth]{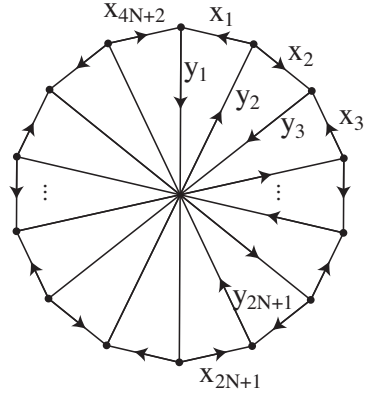}
\caption{An oriented $M_{2N+1}$.}
\label{FigureM_2N+1}
\end{center}
\end{figure}

For any pair of edges $a$ and $b$, let the minimal outer edge distance $d(a,b)$ be defined as the minimum number of edges of any path between $a$ and $b$ using only outer edges (not counting $a$ and $b$).  For $M_{2N+1}$, note that $d(x_i, x_j) \leq 2N$ for any $i, j$.  We define the epsilon coefficient $\varepsilon(a,b)$ of a pair of disjoint edges $a$ and $b$ by:

\begin{equation*}
	\varepsilon(x_i,x_j) = 
	\begin{cases}
		2 & \text{if } d(x_i,x_j) \text{ is odd and } d(x_i,x_j) \neq 2N-1\\
		-1 & \text{if } d(x_i,x_j) = 2N\\
		1 & \text{otherwise}
	\end{cases}
\end{equation*}
\medskip

\begin{equation*}
	\varepsilon(x_i,y_j) = 
	\begin{cases}
		2 & \text{if } d(x_i,y_j) = 1\\
		3 & \text{if } d(x_i,y_j) \geq 2
	\end{cases}
\end{equation*}
\medskip

\begin{equation*}
	\varepsilon(y_i,y_j) = 
	\begin{cases}
		2 & \text{if } d(y_i,y_j) = 1\\
		5 & \text{if } d(y_i,y_j) = 2\\
		6 & \text{if } d(y_i,y_j) \geq 3.
	\end{cases}
\end{equation*}
\medskip

\noindent For any embedding $f: M_{2N+1} \rightarrow {S}^3$ with a regular projection, define:
\medskip

\[
\widehat{L}_{\varepsilon}(f)  = \sum_{a\cap b=\emptyset} \varepsilon(a,b) \ell(f(a),f(b)).
\]
\medskip

\begin{Remark}\rm{ This definition of $\widehat{L}_{\varepsilon}(f)$ does not reduce to the original Simon invariant for $N=1$.}\end{Remark}
\medskip

\begin{Theorem}\label{mobius invariant}
For $N \geq 2$ and any embedding $f$ of $M_{2N+1}$ in ${S}^3$, $\widehat{L}_{\varepsilon}(f) $ is independent of labeling and orientation, and invariant under ambient isotopy of $f(M_{2N+1})$.
\end{Theorem}

\begin{proof}
We first show that $\widehat{L}_{\varepsilon}(f) $ is independent of labeling and orientation.  Since $N \geq 2$, it follows from Simon~\cite{S86} that any automorphism of $M_{2N+1}$ with $N \geq 2$ takes the cycle of outer edges $\overline{x_1 x_2 ... x_{4N+2}}$ to itself, preserving the order of the edges $x_1, x_2, ..., x_{4N+2}$ and thus the edges $y_1, y_2, ..., y_{2N+1}$ as well.  Thus any automorphism either preserves all the arrows in the orientation of $M_{2N+1}$, or reverses all the arrows.  Reversing every arrow would have no effect on the signs of the crossings, so $\widehat{L}_{\varepsilon}(f) $ is independent of labeling and orientation.

As before, it is easy to see that $\widehat{L}_{\varepsilon}(f) $ is invariant under the first four Reidemeister moves.  We show that $\widehat{L}_{\varepsilon}(f) $ is unchanged under the fifth Reidemeister move. Without loss of generality, we may assume that an edge $e$ is pulled over a vertex $v$ and the adjacent outer edges point towards $v$ (see Figure~\ref{FigureM_2N+1_loopover}).  Pulling $e$ over $v$ generates three new crossings: two with outer edges and one with an inner edge.  We must determine the change in $\widehat{L}_{\varepsilon}(f)$ as a result of of these added crossings.

\begin{figure}[here]
\begin{center}
\includegraphics[width=0.3\textwidth]{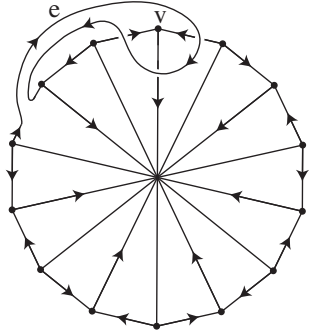}
\caption{$M_{2N+1}$ with an edge $e$ pulled over a vertex $v$.}
\label{FigureM_2N+1_loopover}
\end{center}
\end{figure}

Below we compute the possibilities for this change $\Delta \widehat{L}_{\varepsilon}(f)$,  and show that in all cases this value is zero.  The crossings between the edge $e$ and the two outer edges have the same sign while the crossing of $e$ with an inner edge has the opposite sign.  The epsilon coefficients for the crossings of $e$ with the two outer edges are given in parenthesis (with the edge whose minimal outer edge distance from $e$ is larger given first, and the edge closer to $e$ given second), while the epsilon coefficient of the crossing of $e$ with the inner edge is given afterward. For ease of notation, let $d(e,v)$ denote the minimum number of edges in any path between $e$ and $v$ using only outer edges and not counting $e$.  

\begin{itemize}
\item $e$ is an outer edge
\begin{itemize} 
\item[$\circ$] If $d(e,v) = 1$, then\\
	$\Delta \widehat{L}_{\varepsilon}(f) = (2 + 0) - 2 = 0$
\item[$\circ$] If $d(e,v) = 2, 4, ..., 2N-2$, then\\
	$\Delta \widehat{L}_{\varepsilon}(f) = (1 + 2) - 3 = 0$
\item[$\circ$] If $d(e,v) = 3, 5, ..., 2N-3$, then\\
	$\Delta \widehat{L}_{\varepsilon}(f) = (2 + 1) - 3 = 0$
\item[$\circ$] If $d(e,v) = 2N-1$, then\\
	$\Delta \widehat{L}_{\varepsilon}(f)= (1 + 1) - 2 = 0$
\item[$\circ$] If $d(e,v) = 2N$, then\\
	$\Delta \widehat{L}_{\varepsilon}(f) = (-1 + 1) - 0 = 0$.
\end{itemize}

\item $e$ is an inner edge
\begin{itemize}
\item[$\circ$] If $d(e,v) = 1$, then\\
	$\Delta L(f) = (2 + 0) - 2 = 0$
\item[$\circ$] If $d(e,v) = 2$, then\\
	$\Delta L(f) = (3 + 2) - 5 = 0$
\item[$\circ$] If $d(e,v) \geq 3$, then\\
	$\Delta L(f) = (3 + 3) - 6 = 0$.
\end{itemize}
\end{itemize}
\medskip

Thus $\widehat{L}_{\varepsilon}(f)$ is invariant under the fifth Reidemeister move, and so it is invariant under ambient isotopy\end{proof}

\medskip

It follows that $\widehat{L}_{\varepsilon}(f)$ is a generalized Simon invariant for $M_{2N+1}$.

\medskip

\begin{Lemma}\label{mobius odd}
For any $N \geq 2$ and any embedding $f$ of $M_{2N+1}$ in  ${S}^3$, the generalized Simon invariant $\widehat{L}_{\varepsilon}(f)$ is an odd number.
\end{Lemma}

\begin{proof}
Note that any crossing change of a projection of $f(M_{2N+1})$ will change the signed crossing number between the two edges by $\pm 2$.  Thus any crossing change will alter $\widehat{L}_{\varepsilon}(f)$ by an even number.  Now consider the embedding $f$ of $M_{2N+1}$ shown in Figure~\ref{FigureM2N1_mobius_embedding_labeled}.  There is only one crossing, and it is between two outer edges with an outer edge distance of $2N$ (the maximum).  The epsilon coefficient for this crossing is $-1$, which is multiplied by the crossing sign $-1$ so that $\widehat{L}_{\varepsilon}(f) = (-1)(-1) = 1$ for this embedding.  It follows that $\widehat{L}_{\varepsilon}(f)$ is odd for any embedding of $M_{2N+1}$.\end{proof}

\begin{figure}[here]
\begin{center}
\includegraphics[width=.45\textwidth]{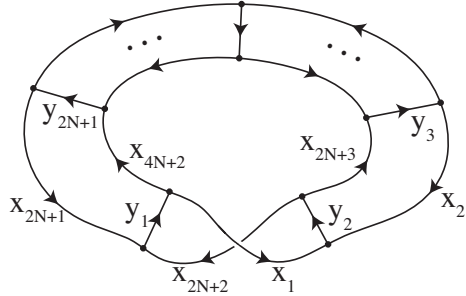}
\caption{An embedding of $M_{2N+1}$ with $\widehat{L}_{\varepsilon}(f) = 1$.}
\label{FigureM2N1_mobius_embedding_labeled}
\end{center}
\end{figure}

\begin{Lemma}\label{mobius automorphism}
Let $N \geq 2$.  If $\alpha$ is an automorphism of $M_{2N+1}$, then the epsilon coefficients of $M_{2N+1}$ and $\alpha(M_{2N+1})$ are the same, and $\alpha$ either preserves the orientation of every edge or reverses the orientation of every edge.
\end{Lemma}

\begin{proof}  Let $\alpha \in$ Aut$(M_{2N+1})$. Since $N \geq 2$, it follows from Simon~\cite{S86} that $\alpha$ takes the cycle $\overline{x_1 x_2 ... x_{4N+2}}$ to itself, preserving the order of the edges $x_1, x_2, ..., x_{4N+2}$ and thus preserving the order of the edges $y_1, y_2, ..., y_{2N+1}$ as well.  Because the order of the outer edges is preserved, the outer edge distance is also preserved.  The epsilon coefficients depend only on the outer edge distance and the distinction between inner and outer edges, so it follows that the epsilon coefficients of $M_{2N+1}$ and $\alpha(M_{2N+1})$ are the same. 

Finally, we can see from Figure~\ref{FigureM_2N+1} that $\alpha$ either preserves all or reverses all the orientations on edges. \end{proof} 
\medskip

To prove that $M_{2N+1}$ is intrinsically chiral, we will use the following Proposition whose proof is similar to that of Theorem~\ref{k7 chiral}.

\begin{Proposition}\label{intrinsic chirality}
Let $G$ be an oriented graph with a generalized Simon invariant $\widehat{L}_{\varepsilon}(f)$.  Suppose that $\widehat{L}_{\varepsilon}(f)$ is odd for every embedding $f: G \rightarrow {S}^3$, and every automorphism of $G$ preserves the epsilon coefficients of $G$ and either preserves the orientation of every edge or reverses the orientation of every edge. Then $G$ is intrinsically chiral.
\end{Proposition}

\begin{proof} For the sake of contradiction, suppose that for some embedding $f$ of $G$, there is an orientation reversing homeomorphism $h$ of the pair (${S}^3$, $f(G)$).  Let $\alpha$ denote the automorphism that $h$ induces on $G$.  

 Fix a sphere of projection $P$ in $S^3$.  Since $f\circ \alpha(G)$ and $f(G)$ are identical as subsets of $S^3$, their projections on $P$ are the same.   Also, since $\alpha$ either preserves all the edge orientations or reverses all the edge orientations, the sign of every crossing in the projection of the oriented embedded graph $h\circ f(G)$ is the same as it is in the projection of the oriented embedded graph $f(G)$.  Furthermore, by hypothesis each crossing has the same epsilon coefficient, whether the crossing is considered in $f\circ \alpha(G)$ or in $f(G)$.  It follows that $\widehat{L}_{\varepsilon}(h^\circ f)=\widehat{L}_{\varepsilon}(\alpha\circ f)=\widehat{L}_{\varepsilon}(f)$.
 
Let $\rho$ denote a reflection of $S^3$ which pointwise fixes the sphere of projection $P$.  Using orientations induced by $G$, the sign of every crossing in the projection of $\rho\circ f(G)$ is the reverse of that of the corresponding crossing in $f(G)$.  It follows that $\widehat{L}_{\varepsilon}(\rho\circ f)=-\widehat{L}_{\varepsilon}(f)$. On the other hand, since $h$ is orientation reversing it is isotopic to $\rho$.  Hence by by definition of a generalized Simon invariant, $\widehat{L}_{\varepsilon}(\rho\circ f)=\widehat{L}_{\varepsilon}(h\circ f)$.  Consequently, $\widehat{L}_{\varepsilon}(h\circ f)=-\widehat{L}_{\varepsilon}(f)$.   Thus $\widehat{L}_{\varepsilon}(f)=0$, which contradicts our hypothesis that $\widehat{L}_{\varepsilon}(f)$ is odd.  Hence $G$ is intrinsically chiral.  \end{proof}

\medskip

 Flapan~\cite{Flapan 1989} showed that $M_{2N+1}$ is intrinsically chiral.  However, now that result follows as an immediate corollary of Lemmas~\ref{mobius odd},~\ref{mobius automorphism}, and Proposition~\ref{intrinsic chirality}.  

\begin{Corollary}\label{mobius chiral}
$M_{2N+1}$ is intrinsically chiral for $N \geq 2$.
\end{Corollary}
\medskip

Nikkuni and Taniyama~\cite{Nikkuni 2009} showed that the Simon invariant provides restrictions on the symmetries of a given embedding of $K_5$ or $K_{3,3}$.  For example, they proved that for both $K_5$ and $K_{3,3}$, the transposition of two vertices can be induced by a homeomorphism on an embedding $f$ only if the Simon invariant of the embedding is $\pm 1$.  By contrast we have the following result for $M_{2N+1}$.

\begin{Theorem}\label{mobius symmetries}  Let $N\geq 2$.
Then for any odd integer m, there is an embedding $f$ of $M_{2N+1}$ in ${S}^3$ with $\widehat{L}_{\varepsilon}(f) = m$ such that every automorphism of $M_{2N+1}$ is induced by a homeomorphism of $({S}^3$, $f(M_{2N+1}))$.
\end{Theorem}

\begin{proof}Let $m$ be an odd integer, and suppose that $|m| = 2k+1$.  Since any automorphism of $M_{2N+1}$ takes the outer loop $x_1 x_2 ... x_{4N+2}$ to itself~\cite{S86}, the automorphism group Aut($M_{2N+1}$) is the dihedral group $D_{2(4N+2)}$.  This group is generated by a rotation of the outer loop of order $4N + 2$ together with a reflection of the outer loop.  Hence, it suffices to show there is an embedding $f: M_{2N+1} \rightarrow {S}^3$ with $\widehat{L}_{\varepsilon}(f) = m$ such that both of the generators of Aut($M_{2N+1}$) are induced by homeomorphisms of (${S}^3$, $f(M_{2N+1})$).

Consider the embedding of $f: M_{2N+1} \rightarrow {S}^3$ shown in Figure~\ref{FigureM2N1vertaxis}. There are $2k+1$ crossings between a pair of outer edges with an outer edge distance of $2N$.  The epsilon coefficient for each of these crossings is $-1$.  If $m > 0$, we embed $M_{2N+1}$ so that all the crossings have negative sign, otherwise embed $M_{2N+1}$ so that the crossings all have positive sign.  Then $\widehat{L}_{\varepsilon}(f) = (-1)(-1)(2k+1) = 2k+1$ if $m > 0$, and $\widehat{L}_{\varepsilon}(f) = (-1)(+1)(2k+1) = -(2k+1)$ if $m < 0$.  Since $|m| = 2k+1$, it follows that $\widehat{L}_{\varepsilon}(f) = m$.

\begin{figure}[here]
\begin{center}
\includegraphics[width=.45\textwidth]{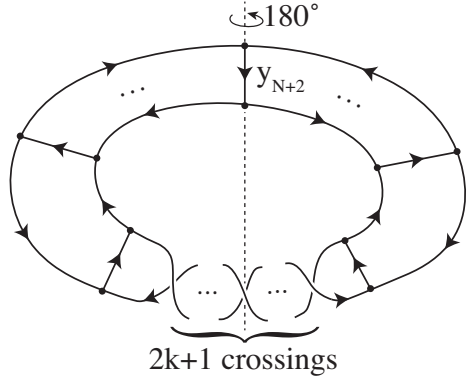}
\caption{An embedding of $M_{2N+1}$ with $\widehat{L}_{\varepsilon}(f) = 2k+1$ with a rotation of the outer loop of order $4N + 2$ together with a reflection of the outer loop.}
\label{FigureM2N1vertaxis}
\end{center}
\end{figure}

By inspection of Figure~\ref{FigureM2N1vertaxis} we see that both generators of Aut($M_{2N+1}$) can be induced by homeomorphisms of (${S}^3$, $f(M_{2N+1})$).\end{proof}

\medskip

Observe that $M_3=K_{3,3}$.  Using our generalized Simon invariant for embeddings of $M_{2N+1}$ $(N \geq 2)$ and the original Simon invariant for embeddings of $M_3$, we now define a topological invariant for embedded Mobius ladders with an even number of rungs (at least 4).  For the remainder of this section, we use $\widehat{L}_{\varepsilon}(f)$ to refer to the Simon invariant if $f$ is an embedding of $M_3$ and to the generalized Simon invariant if $f$ is an embedding of $M_{2N+1}$ for $N \geq 2$.

Let $N \geq 2$ and let $f$ be an embedding of $M_{2N}$ in $S^3$.  For each $i\leq 2N$, let $g_i:M_{2N} \rightarrow S^3$ be the embedding obtained from $f$ by omitting the rung $r_i$ and its vertices from $M_{2N}$.  Note that since $N>1$ the rungs of $M_{2N}$ are setwise invariant under any automorphism~\cite{S86}.  Thus the definition of $g_i$ is unambiguous.  When $N>2$, by Theorem \ref{mobius invariant}, the graph $M_{2N-1}$ has a well defined $\widehat{L}_{\varepsilon}(g_i)$ independent of labeling and orientation.  When $N = 2$, we label each $M_3$ subgraph such that the rungs and outer edges of $M_3$ are contained in the rungs and outer edges of $M_4$ respectively.  Although there are two possible orientations for each embedded $M_3$ subgraph, one can be obtained from the other by reversing the orientation of all edges.  This has no effect on the crossing signs (or epsilon coefficients).  Thus we can unambiguously define:

\[
T_\varepsilon(f) = \sum_{i\leq 2N} \widehat{L}_{\varepsilon}(g_i).
\]

Note that $T_\varepsilon(f)$ is defined on an embedding of the unoriented graph $M_{2N}$.
  \medskip

\begin{Theorem}\label{mobius_even}  For $N\geq 2$ and
any embedding $f$ of $M_{2N}$ in $S^3$, $T_\varepsilon(f)$ is invariant under ambient isotopy.  Furthermore, if $T_\varepsilon(f) \neq 0$, then $f$ is a chiral embedding of $M_{2N}$.
\end{Theorem}

\begin{proof}  By ~\cite{S86}, the cycle of outer edges of $M_{2N}$ is unique.  
Each $\widehat{L}_{\varepsilon}(g_i)$ is invariant under ambient isotopy by Theorem~\ref{mobius invariant} when $N>2$ and by the Simon invariant when $N=2$.  Thus it follows that $T_\varepsilon(f)$ is also invariant under ambient isotopy.  

Let $h$ denote an orientation reversing homeomorphism of $S^3$.  Then $h$ will reverse the signs of all the crossings of $f(M_{2N})$ (and thus each $g_i(M_{2N-1})$).  We now show that the automorphism that $h$ induces on each $M_{2N-1}$ preserves the epsilon coefficients.  If $N\geq3$, then this follows directly from Lemma~\ref{mobius automorphism}.  If instead $N=2$, then by Simon~\cite{S86} the outer edges of $M_{2N}$ are setwise invariant under the automorphism that $h$ induces on $M_{2N}$, so $h$ preserves the distinction between inner and outer edges.  As explained earlier, the edges in each $g_i(M_3)$ subgraph of $f(M_4)$ are labeled as inner or outer in order to match $f(M_4)$.  It follows that $h$ also preserves the distinction between inner and outer edges for each $M_{3}$ subgraph.  For $M_3$, the epsilon coefficients depend only on the distinction between inner and outer edges and on the relative orientation of edges (which is invariant under any automorphism), so the automorphism that $h$ induces on each $M_3$ subgraph preserves the epsilon coefficients.

Since the epsilon coefficients are preserved and the crossings signs are reversed, it follows that each $\widehat{L}_{\varepsilon}(h\circ g_i) = - \widehat{L}_{\varepsilon}(g_i)$ and so $T_\varepsilon(h(f)) = -T(f)$.  If $T_\varepsilon(f) \neq 0$, then $T_\varepsilon(f) \neq -T_\varepsilon(f) = T_\varepsilon(h(f))$, and thus $f(M_{2N})$ is chiral.
\end{proof}
\medskip

\begin{Corollary}
	For all $N \geq 2, m \geq 0$, the embedding $f$ of $M_{2N}$ shown in Figure~\ref{Figurem2n_chiral} is chiral.
\end{Corollary}

\begin{figure}[here]
\begin{center}
\includegraphics[width=.4\textwidth]{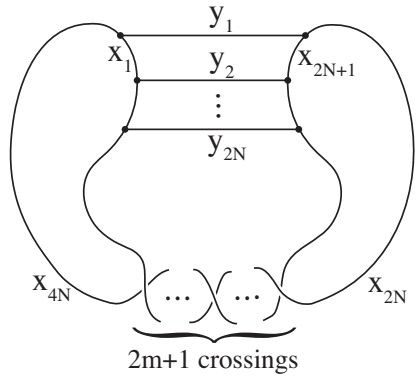}
\caption{An embedding of $M_{2N}$ with $2m+1$ crossings.}
\label{Figurem2n_chiral}
\end{center}
\end{figure}

\begin{proof}
For all of the $f(M_{2N-1})$ subgraphs, the outer edge distance between the two crossed edges is $2N-2$, so the crossing sign and epsilon coefficient for each of the $2m+1$ crossings is the same.  This epsilon coefficient is 1 for $M_3$ when $N=2$, and $-1$ for the generalized Simon invariant (if $N > 2$).  Since all of the $2N$ subgraphs have the same epsilon coefficient and sign for each crossing, both of which are $\pm 1$, it follows that $T_\varepsilon(f) = (\pm 1)(\pm 1)(2N)(2m+1)$ for the embedding $f$ in Figure~\ref{Figurem2n_chiral}.  Since $N \geq 2$ and $m \geq 0$, this means $T_\varepsilon(f) \neq 0$ and thus the embedding is chiral by Theorem~\ref{mobius_even}.
\end{proof}

\medskip

\subsection*{The Heawood graph}

 Let $C_{14}$ denote the Heawood graph oriented and labeled as in Figure~\ref{FigureHeawood}. In particular, we refer to its ``outer edges'' consecutively by $x_1, x_2, ..., x_{14}$, and its ``inner edges'' consecutively by $y_1, y_2, ..., y_{7}$.  We note that this classification of  oriented edges is only dependent on the labeling of the edges in the Hamiltonian cycle $\overline{x_1x_2...x_{14}}$. 

\begin{figure}[here]
\begin{center}
\includegraphics[width=0.4\textwidth]{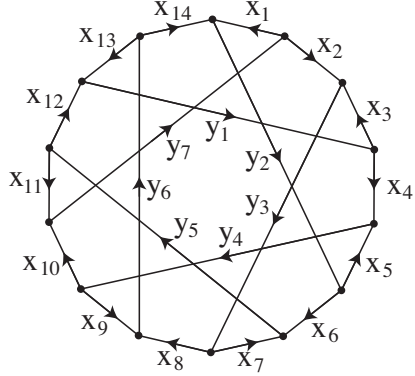}
\caption{An oriented Heawood graph.}
\label{FigureHeawood}
\end{center}
\end{figure}

For any pair of edges $a$ and $b$, let the minimal outer edge distance $d(a,b)$ be defined as the minimum number of edges in any path between $a$ and $b$ using only outer edges (not counting $a$ and $b$).  For any $i, j$, note that $d(x_i, x_j) \leq 6$, $d(x_i, y_j) \leq 4$, and $d(y_i, y_j) \leq 2$.

We define the epsilon coefficient $\varepsilon(a,b)$ of a pair of disjoint edges by:
\begin{equation*}
	\varepsilon(x_i,x_j) = 
	\begin{cases}
		2 & \text{if } d(x_i,x_j) = 1 \text{ or } 4\\
		-2 & \text{if } d(x_i,x_j) = 3 \text{ or } 5 \text{, and } x_i,x_j \text{ are connected by an edge}\\
		-3 & \text{if } d(x_i,x_j) = 5 \text{, and } x_i,x_j \text{ are not connected by an edge}\\
		5 & \text{if } d(x_i,x_j) = 6\\
		1 & \text{otherwise}
	\end{cases}
\end{equation*}
\medskip

\begin{equation*}
	\varepsilon(x_i,y_j) = 
	\begin{cases}
		2 & \text{if } d(x_i,y_j) = 1\\
		3 & \text{if } d(x_i,y_j) = 2 \text{ or } 4\\
		-1 & \text{if } d(x_i,y_j) = 3
	\end{cases}
\end{equation*}
\medskip

\begin{equation*}
	\varepsilon(y_i,y_j) = 
	\begin{cases}
		2 & \text{if } d(y_i,y_j) = 1\\
		5 & \text{if } d(y_i,y_j) = 2.
	\end{cases}
\end{equation*}

\medskip

For any embedding $f: C_{14} \rightarrow {S}^3$ with a regular projection, define
\[
\widehat{L}_{\varepsilon}(f) = \sum_{a\cap b=\emptyset} \varepsilon(a,b) \ell(f(a),f(b)).
\]
\medskip

\begin{Theorem}
\label{thm:Heawood} For any embedding $f$ of $C_{14}$ in ${S}^3$, $\widehat{L}_{\varepsilon}(f)$ is invariant under any ambient isotopy leaving the cycle $\overline{x_1x_2...x_{14}}$ setwise invariant.
\end{Theorem}

\begin{proof}
As demonstrated in the previous proofs, we need only to verify that $L$ is invariant under the fifth Reidemeister move.  It suffices to show that $\widehat{L}_{\varepsilon}(f)$ is unchanged when any of the 21 edges in the Heawood graph is pulled over a particular vertex.  This is easy to check using the method shown in the proof of Theorem~\ref{mobius invariant}.\end{proof}

\medskip

It follows that  $\widehat{L}_{\varepsilon}(f)$ is a generalized Simon invariant of $C_{14}$.
\medskip

\begin{Lemma}
\label{thm:Heawood odd} For any embedding $f$ of $C_{14}$ in ${S}^3$, the generalized Simon invariant $\widehat{L}_{\varepsilon}(f)$ is an odd number.
\end{Lemma}

\begin{proof}
Since any crossing change will change the signed crossing number between the two edges by $\pm 2$, we only need to find an embedding $f$ where $\widehat{L}_{\varepsilon}(f)$ is odd.  Consider an embedding of the Heawood graph which has Figure~\ref{FigureHeawood} as its projection with the intersections between edges replaced by crossings.  The reader can check that regardless of the signs of the crossings, there are an odd number of crossings with odd epsilon coefficient.  Hence $\widehat{L}_{\varepsilon}(f)$ is an odd number.
\end{proof}

The proof of the following lemma is left as an exercise.

\begin{Lemma}\label{Heawood automorphism}
Let $\alpha$ be an automorphism of $C_{14}$ that takes the Hamiltonian cycle $\overline{x_1x_2...x_{14}}$ to itself.  Then corresponding epsilon coefficients of $C_{14}$ and $\alpha(C_{14})$ are equal, and $\alpha$ either preserves the orientation of every edge or reverses the orientation of every edge.
\end{Lemma}

\begin{Theorem}\label{Heawood chiral}
The Heawood graph is intrinsically chiral.
\end{Theorem}

 \begin{proof}  Let $C_{14}$ denote the Heawood graph.  Suppose that for some embedding $f$ of $C_{14}$ in $S^3$, there is an orientation reversing homeomorphism $h$ of $(S^3, f(C_{14}))$.  It was shown by Nikkuni \cite{Nikkuni 2012} that the mod 2 sum of the Arf invariants of all the 14-cycles and 12-cycles in an embedding of $C_{14}$ is 1.  Thus $f(C_{14})$ either has an odd number of 14-cycles with Arf invariant 1 or an odd number of 12-cycles with Arf invariant 1.   By arguing as in the proof of Corollary \ref{K4n+3}, without loss of generality we can assume that the order of the automorphism that $h$ induces on $C_{14}$ is a power of 2.  It follows that $h$ either leaves some 14-cycle or some 12-cycle setwise invariant.
 
 Suppose that $h$ leaves a 14-cycle setwise invariant.  Label the edges of this 14-cycle consecutively as $\overline{x_1x_2...x_{14}}$.  Then it follows from Lemma \ref{Heawood automorphism}, that $\widehat{L}_{\varepsilon}(h\circ f)=\widehat{L}_{\varepsilon}(f)$.  But since $h$ is orientation reversing we can argue as in the proof of Proposition~\ref{intrinsic chirality} that $\widehat{L}_{\varepsilon}(h\circ f)=-\widehat{L}_{\varepsilon}(f)$, which is impossible since $L(f)$ is odd and hence non-zero.
 
 Now suppose that $h$ leaves a 12-cycle $Z$ setwise invariant.  As shown in Figure~\ref{Figure12cycle}, $G$ has precisely three edges not in $Z$ which have both vertices in $Z$.  Now $Z$ together with these three edges is a M\"{o}bius ladder $M_3$.  However, it was shown in \cite{Flapan 1989} that no embedding of $M_3$ in $S^3$ has an orientation reversing homeomorphism which takes the outer loop $Z$ to itself.  Thus again we have a contradiction. \end{proof}

\begin{figure}[here]
\begin{center}
\includegraphics[width=0.4\textwidth]{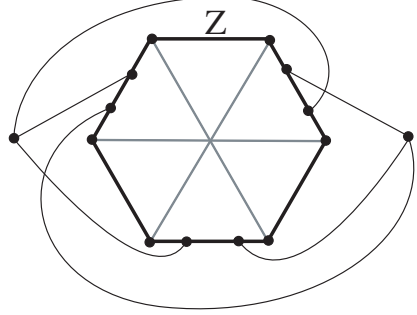}
\caption{A 12-cycle in the Heawood graph.}
\label{Figure12cycle}
\end{center}
\end{figure}

\section{The subgraphs $2K_3$, $K_5$, and $K_{3,3}$ of a given graph}

Shinjo and Taniyama \cite{Shinjo} proved that two embeddings $f$ and $g$ of a graph $G$ in $S^3$ are spatial-graph homologous
if and only if for each $2K_3$ subgraph $H$ of $G$ the restriction maps $f |_H$ and $g|_H$ have the same linking number, and for each $K_5$ or $K_{3,3}$ subgraph $H$ of $G$ the restriction
maps $f |_H$ and $g|_H$ have the same Simon invariant.  

We now show that for any oriented graph $G$, any integer linear combination of
the reduced Wu invariants of subgraphs of $G$ is itself a reduced Wu invariant for $G$.
\medskip

\begin{Theorem}\label{linear} 
Let $G$ be a graph with oriented edges, and let $G_{1},G_{2},\ldots,G_{k}$ denote subgraphs of $G$ with orientations inherited from $G$. For each $q\le k$, let $\varepsilon_{q}:L(G_{q})\to {\mathbb Z}$ be a homomorphism, and $i_{q}:G_{q}\to G$ be the inclusion map. Let $m_{1},m_{2},\ldots,m_{k}$ be integers and let $\varepsilon:L(G) \to {\mathbb Z}$ be the homomorphism given by $\varepsilon=\sum_{q=1}^{k}m_{q}\varepsilon_{q}\circ (i_{q}\times i_{q})^{*}$. Then for any embedding $f$ of $G$ in $S^{3}$, $\sum_{q=1}^{k}m_{q}\tilde{\mathcal L}_{\varepsilon_{q}}(f|_{G_{q}})$ is the reduced Wu invariant given by $\tilde{\mathcal L}_{\varepsilon}(f)$. \end{Theorem}

\begin{proof}
Observe that the embedding $(f\times f)\circ (i_{q}\times i_{q})$ is equivalent to the embedding $(f|_{G_{q}})\times (f|_{G_{q}}):C_{2}(G_{q})\to C_{2}({\mathbb R}^{3})$.  Hence, by the definition of the Wu invariant, it follows that 
\begin{eqnarray*}
\tilde{\mathcal L}_{\varepsilon}(f)
&=& \varepsilon({\mathcal L}(f))\\
&=& \varepsilon((f\times f)^{*}(\Sigma)) \\
&=& \sum_{q=1}^{k}m_{q}\varepsilon_{q}\circ (i_{q}\times i_{q})^{*}\circ (f\times f)^{*}(\Sigma) \\
&=& \sum_{q=1}^{k}m_{q}\varepsilon_{q}\circ ((f\times f)\circ (i_{q}\times i_{q}))^{*}(\Sigma)\\
&=& \sum_{q=1}^{k}m_{q}\varepsilon_{q}\circ ((f|_{G_{q}})\times (f|_{G_{q}}))^{*}(\Sigma)\\
&=& \sum_{q=1}^{k}m_{q}\varepsilon_{q}({\mathcal L}(f|_{G_{q}})) \\
&=& \sum_{q=1}^{k}m_{q}\tilde{\mathcal L}_{\varepsilon_{q}}(f|_{G_{q}}). 
\end{eqnarray*}
Thus we have the result. 
\end{proof}

\medskip

This theorem allows us to define new reduced Wu invariants, as we see from the following two examples.

\medskip

\begin{Example}\label{ML}
{\rm For $N\geq 2$, consider the oriented labeled graph of a M\"{o}bius ladder $M_{2N+1}$ illustrated in Figure~\ref{Mobius}.   For $q=0,1,\ldots,2N$, let $G_{q}$ be the subgraph of $M_{2N+1}$ consisting of the outer cycle $\overline{x_1 x_2 ... x_{4N+2}}$ together with the three rungs $y_{q+1}$, $y_{q+2}$, and $y_{q+3}$ where the subscripts are considered mod $2N+1$ and the orientations are inherited from $M_{2N+1}$.  Then each $G_{q}$ is homeomorphic to $K_{3,3}$. Thus each $L(G_{q})$ is generated by $[E^{x_{q+1},x_{q+2N+2}}]$. Let $\varepsilon_{q}$ be the homomorphism from $L(G_{q})$ to ${\mathbb Z}$ defined by $\varepsilon_{q}(x_{q+1},x_{q+2N+2})=1$.  Let $f$ be an embedding of $M_{2N+1}$ in $S^3$.  Then by Theorem~\ref{linear}, $\tilde{\mathcal L}_\varepsilon(f)=\sum_{q=0}^{2N}\tilde{\mathcal L}_{\varepsilon_{q}}(f|_{G_{q}})$ defines a reduced Wu invariant for $M_{2N+1}$.

Observe that this reduced Wu invariant is not equal to the generalized Simon invariant for $M_{2N+1}$ that we defined in Section~\ref{GSI}.  However, this invariant has similar properties to those we proved for the generalized Simon invariant of $M_{2N+1}$.  In particular, since each $\tilde{\mathcal L}_{\varepsilon_{q}}(f|_{G_{q}})$ is essentially the Simon invariant of $f|_{G_{q}}$ and therefore odd valued, it follows that $\tilde{\mathcal L}(f)$ is always odd. Moreover, we know from \cite{S86} that any automorphism of $M_{2N+1}$ that takes the outer cycle $\overline{x_1 x_2 ... x_{4N+2}}$ to itself. Thus any automorphism of $M_{2N+1}$ leaves $\left\{G_{0},G_{1},\ldots,G_{2N}\right\}$ setwise invariant. This implies that $\tilde{\mathcal L}(f)$ is independent of labeling. }
\end{Example}

\begin{figure}[htbp]
      \begin{center}
{\includegraphics[width=0.95\textwidth]{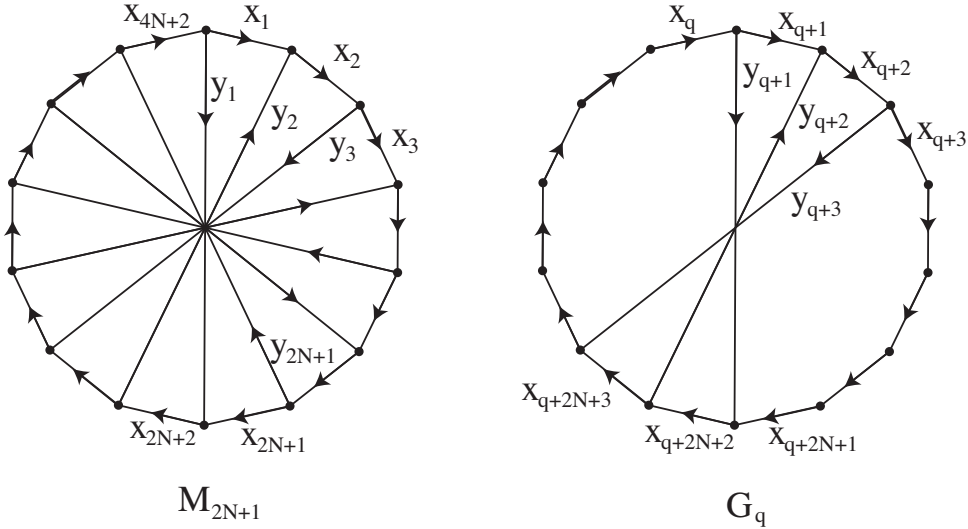}}
      \end{center}
   \caption{An oriented $M_{2N+1}$ together with a $K_{3,3}$ subgraph.  Note the subscripts on $x_i$ are considered mod $4N+2$ and those on $y_i$ are considered mod $2N+1$.}
  \label{Mobius}
\end{figure} 

\begin{Example}\label{Hea}
{\rm 
Let $C_{14}$ be the Heawood graph as illustrated in Figure \ref{Heawood}. For $q=0,1,\ldots,6$, let $G_{q}$ be the subgraph of $C_{14}$ as illustrated in Figure \ref{Heawood}, where the labels of vertices are considered mod 14. Note that each $G_{q}$ is homeomorphic to $K_{3,3}$. Thus each $L(G_{q})$ is generated by $[E^{x_{1},x_{8}}]$. Let $\varepsilon_{q}$ be the homomorphism from $L(G_{q})$ to ${\mathbb Z}$ defined by $\varepsilon_{q}(x_{1},x_{8})=1$.  Let $f$ be an embedding of $C_{14}$ in $S^3$.  Then by Theorem~\ref{linear}, $\tilde{\mathcal L}_\varepsilon(f)=\sum_{q=0}^{6}\tilde{\mathcal L}_{\varepsilon_{q}}(f|_{G_{q}})$ defines a reduced Wu invariant for $C_{14}$.

Again this reduced Wu invariant is not equal to the generalized Simon invariant for $C_{14}$ that we defined in Section~\ref{GSI}, but has similar properties to those of the generalized Simon invariant.  In particular, since each $\tilde{\mathcal L}_{\varepsilon_{q}}(f|_{G_{q}})$ is essentially the Simon invariant of $f|_{G_{q}}$ and therefore odd valued, it follows that $\tilde{\mathcal L}(f)$ is always odd.  Moreover, let $\alpha$ be an automorphism of $C_{14}$ takes the outer cycle $\overline{x_1\dots x_{14}}$ to itself, and thus the edges $y_{1},y_{2},\ldots,y_{7}$ as well. Then $\alpha$ permutes $\left\{G_{0},G_{1},\ldots,G_{6}\right\}$ and reversing every arrow would have no effect on the signs of the crossings. This implies that $\tilde{\mathcal L}(f)$ is preserved under $\alpha$. 
}
\end{Example}

\begin{figure}[htbp]
      \begin{center}
{\includegraphics[width=0.9\textwidth]{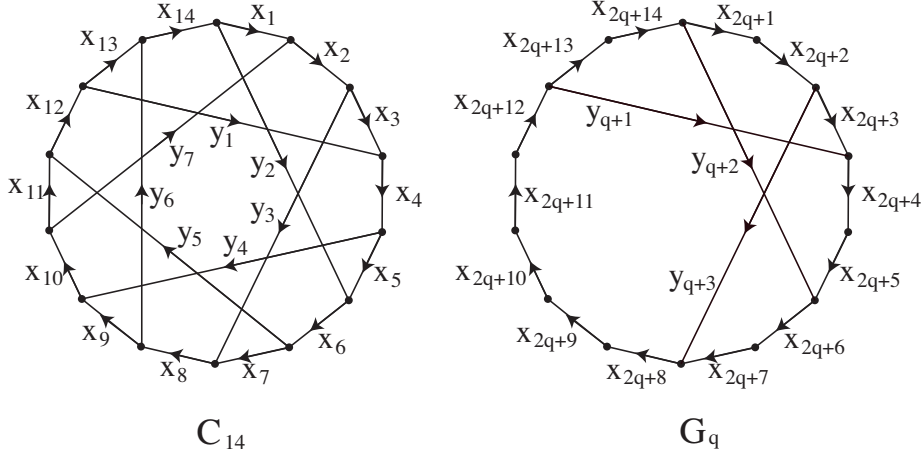}}
      \end{center}
   \caption{An oriented Heawood graph together with a $K_{3,3}$ subgraph.}
  \label{Heawood}
\end{figure}

Now we prove the converse of Theorem~\ref{linear}.  In particular, we show that any reduced Wu invariant of a graph $G$ can be expressed as a linear combination of reduced Wu invariants of subgraphs $2K_3$, $K_5$, and $K_{3,3}$ of $G$.
\medskip

\begin{Theorem}\label{reduced}  Let $G$ be a graph with oriented edges, and let $G_{1},G_{2},\ldots,G_{k}$ denote all of the $2K_{3}$, $K_{5}$, and $K_{3,3}$ subgraphs of $G$ with orientations inherited from $G$.  For each $q\leq k$, let $\varepsilon_{q}:L(G_q)\to {\mathbb Z}$ be an isomorphism, and let $i_{q}:G_q\to G$ be the inclusion map. Then for any homomorphism $\varepsilon: L(G)\to {\mathbb Z}$, there exists integers $m\neq 0$ and $m_{1},m_{2},\ldots,m_{k}$ such that for any embedding $f$ of $G$ in $S^3$.
\begin{eqnarray*}
m\tilde{\mathcal L}_{\varepsilon}(f)
=\sum_{q=1}^{k}m_{q}\tilde{\mathcal L}_{\varepsilon_{q}}(f|_{G_{q}}). 
\end{eqnarray*}

\end{Theorem}

\begin{proof}Consider the homomorphism
\begin{eqnarray*}
\varphi:L(G)\longrightarrow \bigoplus_{q=1}^{k}L(G_{q})
\end{eqnarray*}
defined by 
\begin{eqnarray*}
\varphi(x)=((i_{1}\times i_{1})^{*}(x),(i_{2}\times i_{2})^{*}(x),\ldots,(i_{k}\times i_{k})^{*}(x))
\end{eqnarray*}
 Shinjo and Taniyama \cite{Shinjo} proved that for any $x,y\in L(G)$, if $(i_{q}\times i_{q})^{*}(x)=(i_{q}\times i_{q})^{*}(y)$ for any $q=1,2,\ldots,k$ then $x=y$. This implies that $\varphi$ is injective. It follows that $\varphi$ also induces an injective linear map 
\begin{eqnarray*}
\varphi:L(G)\otimes {\mathbb Q}\longrightarrow \bigoplus_{q=1}^{k}(L(G_{q})\otimes {\mathbb Q})
\end{eqnarray*}
and therefore its dual 
\begin{eqnarray*}
{\varphi}^{\sharp}:{\rm Hom}\left(\bigoplus_{q=1}^{k}(L(G_{q})\otimes {\mathbb Q}),{\mathbb Q}\right)\longrightarrow {\rm Hom}(L(G)\otimes {\mathbb Q},{\mathbb Q})
\end{eqnarray*}
is surjective. We consider each $\varepsilon_{q}$ as a linear map from $\bigoplus_{q=1}^{k}(L(G_{q})\otimes {\mathbb Q})$ to ${\mathbb Q}$ in the usual way. Then because each $\varepsilon_{q}$ is an isomorphism, the linear forms $\varepsilon_{1}$, $\varepsilon_{2}$, \dots, $\varepsilon_{k}$ generate ${\rm Hom}\left(\bigoplus_{q=1}^{k}(L(G_{q})\otimes {\mathbb Q}),{\mathbb Q}\right)$. Thus, for any $u\in {\rm Hom}(L(G)\otimes {\mathbb Q},{\mathbb Q})$, there is a $u'\in {\rm Hom}\left(\bigoplus_{q=1}^{k}(L(G_{q})\otimes {\mathbb Q}),{\mathbb Q}\right)$ and rational numbers $r_{1},r_{2},\ldots,r_{k}$ such that $u'=\sum_{q=1}^{k}r_{q}\varepsilon_{q}$. Hence for an element $x$ in $L(G)\otimes {\mathbb Q}$, we have  
\begin{eqnarray*}
u(x)&=&{\varphi}^{\sharp}(u')(x)\\
&=&{\varphi}^{\sharp}\left(\sum_{q=1}^{k}r_{q}\varepsilon_{q}\right)(x)\\
&=& \sum_{q=1}^{k}r_{q}{\varphi}^{\sharp}(\varepsilon_{q})(x)\\
&=& \sum_{q=1}^{k}r_{q}\varepsilon_{q}(\varphi(x))\\
&=& \sum_{q=1}^{k}r_{q}\varepsilon_{q}\circ (i_{q}\times i_{q})^{*}(x). 
\end{eqnarray*}
Now it follows that $\varepsilon_{1}\circ (i_{1}\times i_{1})^{*},\ \varepsilon_{2}\circ (i_{2}\times i_{2})^{*},\ \ldots,\ \varepsilon_{k}\circ (i_{k}\times i_{k})^{*}$ generate ${\rm Hom}(L(G)\otimes {\mathbb Q},{\mathbb Q})$.  Hence, there are rational numbers $r_{1},r_{2},\ldots,r_{k}$ such that 
\begin{eqnarray*}
\tilde{\mathcal L}_{\varepsilon}(f)
= \sum_{q=1}^{k}r_{q}\tilde{\mathcal L}_{\varepsilon_{q}}(f|_{G_{q}}). 
\end{eqnarray*}
This implies the desired conclusion. 
\end{proof}

\medskip

\begin{Example}\label{K6gs2}\rm{Consider the oriented and labeled $K_6$ illustrated in Figure~\ref{K6Simon3}.  Let $\varepsilon$ be the homomorphism from $L(K_{6})$ to ${\mathbb Z}$ given in Example \ref{K6gs}, and let $\tilde{\mathcal L}_{\varepsilon}(f)
$ be the corresponding reduced Wu invariant.  For $q=1$,\dots, 6, let $G_{q}$ be the $K_{5}$ subgraphs illustrated in Figure \ref{K6Simon3} where $q$ is considered mod 6.   Observe that the orientations and labels on $G_q$ are inherited from those on $K_6$.  Then for each $q$, the linking module $L(G_{q})$ is generated by $[E^{x_{1},x_{4}}]$. Let $\varepsilon_{q}$ be the isomorphism from $L(G_{q})$ to ${\mathbb Z}$ defined by $\varepsilon_{q}(x_{1},x_{4})=1$.  Let $f$ be an embedding of $K_6$ in $S^3$.  Then it's not hard to check that: }
\begin{eqnarray*}
2\tilde{\mathcal L}_{\varepsilon}(f)
=\sum_{q=1}^{6}\tilde{\mathcal L}_{\varepsilon_{q}}(f|_{G_{q}}). 
\end{eqnarray*}
\end{Example}

\begin{figure}[htbp]
      \begin{center}
{\includegraphics[width=\textwidth]{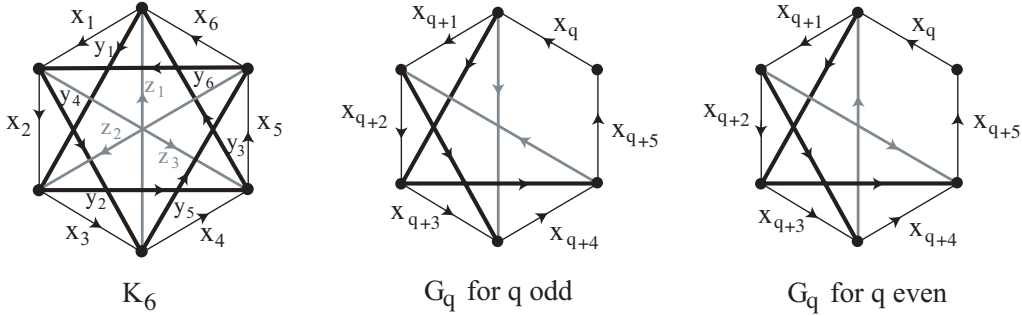}}
      \end{center}
   \caption{An oriented $K_6$ together with $K_5$ subgraphs.}
  \label{K6Simon3}
\end{figure} 

\medskip

\begin{Example}\label{K7gs2}\rm{Consider the oriented and labeled $K_7$ illustrated in Figure~\ref{FigureK7_labeled}.  The epsilon coefficients which gave us the generalized Simon invariant for $K_7$ are
$$	\varepsilon(x_i,x_j) = \varepsilon(y_i,y_j) = \varepsilon(z_i,z_j) = \varepsilon(x_i,z_j) = \varepsilon(y_i,z_j) = 1 	$$
$$	\varepsilon(x_i,y_j) = -1	$$

\noindent These values of $\varepsilon(a,b)$ define a homomorphism $\varepsilon: L(K_{7})\to {\mathbb Z}$, which corresponds to a reduced Wu invariant $\tilde{\mathcal L}_{\varepsilon}(f)$.  For $q=1,2,\ldots,7$, let $G_{q},H_{q},F_{q}$ and $L_{q}$ be the subgraphs of $K_{7}$ illustrated in Figure \ref{K7Simon4}, where the subscripts are considered mod 7.  Observe that the orientations on the subgraphs are inherited from those of $K_7$ in Figure \ref{FigureK7_labeled}.

\begin{figure}[htbp]
      \begin{center}
{\includegraphics[width=1\textwidth]{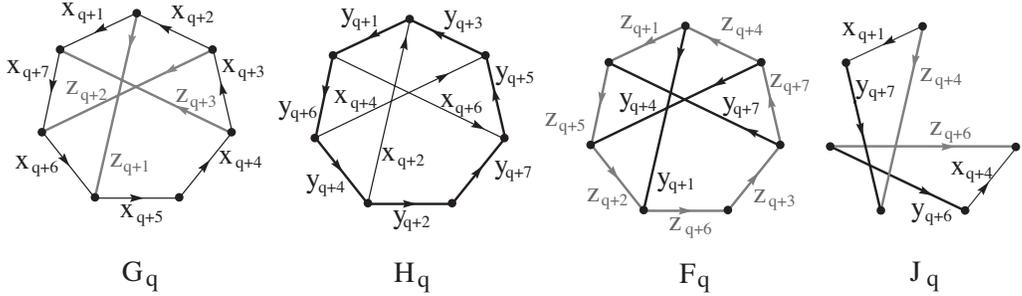}}
      \end{center}
   \caption{We consider these oriented subgraphs of $K_7$.}
  \label{K7Simon4}
\end{figure} 

Each $G_{q},H_{q}$, and $F_{q}$ is homeomorphic to $K_{3,3}$.  Each $L(G_{q})$ is generated by $[E^{x_{1},x_{4}}]$, each $L(H_{q})$ is generated by $[E^{y_{1},y_{7}}]$ and each $L(F_{q})$ is generated by $[E^{z_{1},z_{3}}]$. On the other hand, each $J_{q}$ is homeomorphic to $2K_{3}$, and each $L(J_q)$ is generated by $[E^{x_{q+1},x_{q+4}}]$. Let $\varepsilon_{q}$ be the homomorphism from $L(G_{q})$ to ${\mathbb Z}$ defined by $\varepsilon_{q}(x_{1},x_{4})=1$. Let $\zeta_{q}$ be the homomorphism from $L(H_{q})$ to ${\mathbb Z}$ defined by $\zeta_{q}(y_{1},y_{7})=1$. Let $\eta_{q}$ be the homomorphism from $L(F_{q})$ to ${\mathbb Z}$ defined by $\eta_{q}(z_{1},z_{3})=1$. Let $\theta_{q}$ be the homomorphism from $L(J_{q})$ to ${\mathbb Z}$ defined by $\theta_{q}(x_{q+1},x_{q+4})=1$. Let $f$ be an embedding of $K_7$ in $S^3$. Then it is not hard to check that:}
\begin{eqnarray*}
3\tilde{\mathcal L}_{\varepsilon}(f)
=\sum_{q=1}^{7}\tilde{\mathcal L}_{\varepsilon_{q}}(f|_{G_{q}})
+\sum_{q=1}^{7}\tilde{\mathcal L}_{\zeta_{q}}(f|_{H_{q}})
+\sum_{q=1}^{7}\tilde{\mathcal L}_{\eta_{q}}(f|_{F_{q}})
-5\sum_{q=1}^{7}\tilde{\mathcal L}_{\theta_{q}}(f|_{J_{q}}). 
\end{eqnarray*}
\end{Example}

\medskip 

\section{Minimal crossing number of a spatial graph} 

Let $f$ be a spatial embedding of a graph $G$. The following theorem gives a lower bound for the minimal crossing number of any projection of $f$ up to isotopy.

\begin{Theorem}\label{cr}Let $f$ be an embedding of an oriented graph $G$ in $S^3$ with generalized Simon invariant $\widehat{L}_{\varepsilon}(f)$, and let $c(f)$ be the minimum crossing number of all projections of all embeddings ambient isotopic $f$.  Let $m_{\varepsilon}$ be the maximum of $|\varepsilon(e_{i},e_{j})|$ over all pairs of disjoint edges in $G$.  Then $$\left|\widehat{L}(f)\right| \le c(f) m_{\varepsilon}.$$ 
\end{Theorem}

\begin{proof}  Fix a diagram of $f(G)$ which realizes the minimal crossing number $c(f)$.  Observe that $c(f)$ includes  crossings between an edge and itself as well as crossings between adjacent edges, which are not included in $\sum_{e_{i}\cap e_{j}=\emptyset}|\ell(f(e_{i}),f(e_{j}))|$.  Therefore, we have the following sequence of inequalities.
\begin{eqnarray*}
|\widehat{L}_{\varepsilon}(f)| 
&=& \left|\sum_{e_{i}\cap e_{j}=\emptyset}\varepsilon(e_{i},e_{j})\ell(f(e_{i}),f(e_{j}))\right|\\
&\le& \sum_{e_{i}\cap e_{j}=\emptyset}|\varepsilon(e_{i},e_{j})| |\ell(f(e_{i}),f(e_{j}))|\\
&\le& m_{\varepsilon} \sum_{e_{i}\cap e_{j}=\emptyset}|\ell(f(e_{i}),f(e_{j}))|\\
&\le& m_{\varepsilon} c(f). 
\end{eqnarray*}
Thus we have the result. 
\end{proof}

Since every reduced Wu invariant with respect to a given homomorphism $\varepsilon$ is a generalized Simon invariant with epsilon coefficients given by $\varepsilon(a,b)$, Theorem \ref{cr} is true for any reduced Wu invariant $\tilde{\mathcal L}_{\varepsilon}(f)$.  

Recall from Example \ref{2K3_2} that the reduced Wu invariant of $2K_3$ is twice the linking number.  Thus applying Theorem \ref{cr} to an embedding of $2K_3$ gives us the well known fact that the minimal crossing number of a $2$-component link is at least twice the absolute value of the linking number.  Applying Theorem \ref{cr} to Examples \ref{K5} and \ref{K33} shows that the minimal crossing number of any spatial embedding of $K_5$ or $K_{3,3}$ is at least the absolute value of the Simon invariant.

\begin{Example}\rm{
Let $f$ be a spatial embedding of $K_{7}$. Consider the generalized Simon invariant $\widehat{L}_{\varepsilon}(f)$ given in Section~\ref{GSI}. Since $m_{\varepsilon}(f)=1$ for any projection of $f$, it follows from Theorem \ref{cr} that $c(f)\ge |\widehat{L}_{\varepsilon}(f)|$. }
\end{Example}

\begin{Example}\rm{Consider the oriented and labeled $K_6$ illustrated in Figure~\ref{K6Simon3}. 
  We introduce a new generalized Simon invariant for $K_6$ where the epsilon coefficients are given by:

$$\varepsilon(x_i,x_j)=\varepsilon (z_i,z_j)=1$$

$$\varepsilon(y_i,y_j)=\varepsilon(x_i, z_j)=-1$$

$$\varepsilon(x_i, y_j)=\varepsilon(y_i, z_j)=0$$
\medskip

It is not hard to check that these epsilon coefficients indeed give us a generalized Simon invariant for $K_6$.  Alternatively, if we let $T_1$ be the triangle with vertices $y_1$, $y_2$, and $y_3$ and let $T_2$ be the triangle with vertices $y_4$, $y_5$, and $y_6$, then we can define $\tilde{\mathcal L}_{\varepsilon}(f)$ as the sum of $2\mathrm{lk}(f(T_1), f(T_2))$ together with the Simon invariant of the oriented $K_{3,3}$-subgraph obtained from $K_6$ by deleting $T_1$ and $T_2$.

Let $f$ be the spatial embedding of $K_{6}$ illustrated in Figure \ref{K6Simon2}, where the rectangle represents the number of positive crossings. We compute the generalized Simon invariant $\tilde{\mathcal L}_{\varepsilon}(f)$ as: 

\begin{eqnarray*}
\tilde{\mathcal L}_{\varepsilon}(f)=\varepsilon(y_{1},y_{4})\cdot (2n+1)+\varepsilon(x_{4}, z_{2})\cdot 1+\varepsilon(y_{6}, y_{3})\cdot 1=-(2n+1)-1=-2n-3\end{eqnarray*}

Since $m_{\varepsilon}=1$, it follows from Theorem~\ref{cr} that $c(f)\geq|\widehat{L}_{\varepsilon}(f)|=2n+3$.  The projection in Figure~\ref{K6Simon2} has $2n+3$ crossings.  Thus this projection has a minimal number of crossings.  In particular, this means that for every odd number $k\geq 3$, there is an embedding $g$ of $K_6$ in $S^3$ such that $c(g)=k$.}

\end{Example}

\begin{figure}[h]
      \begin{center}
\scalebox{0.45}{\includegraphics*{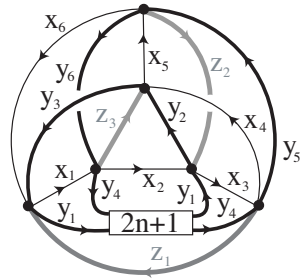}}
      \end{center}
   \caption{This projection of an embedded $K_6$ has a minimal number of crossings.}
  \label{K6Simon2}
\end{figure} 

\begin{Example}\rm{Let $f$ be the embedding of the Heawood graph illustrated in Figure~\ref{HeawoodEmbed}, where each of the rectangles represent the  number of positive crossings.  Using the generalized Simon invariant from Section~\ref{GSI}, we find that $\widehat{L}_{\varepsilon}(f)=5(2k+1)+5(2m+1)+5(2n+1)$.  Also, $m_\varepsilon=5$.  Now it follows from Theorem~\ref{cr} that $c(f)\geq 2(k+m+n)+3$.  Since this is precisely the number of crossings in Figure ~\ref{HeawoodEmbed}, it follows that this projection has a minimal number of crossings.  Since we can choose any values for $k$, $m$, and $n$, it follows that for every odd number $l\geq 3$, there is an an embedding $g$ of the Heawood graph in $S^3$ such that $c(g)=l$.

}\end{Example}

\begin{figure}[h]
      \begin{center}
\scalebox{0.8}{\includegraphics*{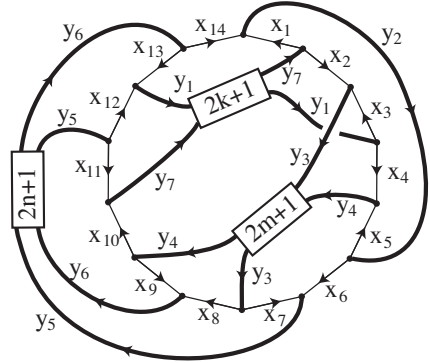}}
      \end{center}
   \caption{This projection of an embedded Heawood graph has a minimal number of crossings.}
  \label{HeawoodEmbed}
\end{figure} 

\begin{acknowledgements}

The authors are grateful to Professor Kouki Taniyama for suggesting that the
Wu invariant might be used to obtain bounds on the minimal crossing number
of a spatial graph. 

The first author was supported in part by NSF grant DMS-0905087, and the third author was partially supported by Grant-in-Aid for Scientific Research (C) (No. 21740046), Japan Society for the Promotion of Science.
Also, the first author  thanks the Institute for Mathematics and its Applications at the University of Minnesota for its hospitality during the Fall of 2013, when she was a long term visitor.
\end{acknowledgements}

\renewcommand\bibname{References}

\end{document}